\documentclass[reqno,12pt]{amsart}
\usepackage[colorlinks=true, linkcolor=blue, citecolor=blue]{hyperref}

\usepackage{amssymb}
\usepackage{amsmath, graphicx, rotating}
\usepackage{color}
\usepackage{soul}
\usepackage[dvipsnames]{xcolor}

\usepackage{ifthen}
\usepackage{xkeyval}
\usepackage{todonotes}
\usepackage[T1]{fontenc}
\usepackage{lmodern}
\usepackage[english]{babel}

\usepackage{ upgreek }
\usepackage{stmaryrd}
\SetSymbolFont{stmry}{bold}{U}{stmry}{m}{n}
\usepackage{amsthm}
\usepackage{float}

\usepackage{ bbm }
\usepackage{ stmaryrd }
\usepackage{ mathrsfs }
\usepackage{ frcursive }
\usepackage{ comment }

\usepackage{pgf, tikz}
\usetikzlibrary{shapes}
\usepackage{varioref}
\usepackage{enumitem}

\definecolor{rouge}{rgb}{0.7,0.00,0.00}
\definecolor{vert}{rgb}{0.00,0.5,0.00}
\definecolor{bleu}{rgb}{0.00,0.00,0.8}
\usepackage[margin=1in]{geometry}
\newtheorem{theorem}{Theorem}[section]
\newtheorem*{theorem*}{Theorem}
\newtheorem{lemma}[theorem]{Lemma}

\newtheorem{proposition}[theorem]{Proposition}

\labelformat{hypothesis}{\textbf{M\kern-0.1mm#1}}

\newtheorem{condition}{Condition}

\labelformat{conditionA}{\textbf{A\kern-0.1mm#1}}

\theoremstyle{definition}
\newtheorem{example}[theorem]{Example}
\newtheorem{remark}[theorem]{Remark}

\def \eref#1{\hbox{(\ref{#1})}}

\numberwithin{equation}{section}

\def\bf#1{\mathbf{#1}}

\def\geq{\geqslant}
\def\leq{\leqslant}

\def\RR{\mathbb{R}}
\def\PP{\mathbb{P}}
\def\EE{\mathbb{E}}

\def \eref#1{\hbox{(\ref{#1})}}

\def\EE{\mathbb{ E}}

\def\ep{{\epsilon}}

\begin{document}

\title[Averaging principle for SDEs]
{ Averaging principle for slow-fast stochastic differential equations with time dependent locally Lipschitz coefficients}

\author{Wei Liu}
\curraddr[Liu, W.]{ School of Mathematics and Statistics, Jiangsu Normal University, Xuzhou, 221116, China}
\email{weiliu@jsnu.edu.cn}

\author{Michael R\"{o}ckner}
\curraddr[R\"{o}ckner, M.]{Fakult\"{a}t f\"{u}r Mathematik, Universit\"{a}t Bielefeld, D-33501 Bielefeld, Germany, and Academy of Mathematics and Systems Science,
  Chinese Academy of Sciences (CAS), Beijing, 100190, P.R.China}
\email{roeckner@math.uni-bielefeld.de}

\author{Xiaobin Sun}
\curraddr[Sun, X.]{ School of Mathematics and Statistics, Jiangsu Normal University, Xuzhou, 221116, China}
\email{xbsun@jsnu.edu.cn}

\author{Yingchao Xie}
\curraddr[Xie, Y.]{ School of Mathematics and Statistics, Jiangsu Normal University, Xuzhou, 221116, China}
\email{ycxie@jsnu.edu.cn}

\begin{abstract}
This paper is devoted to studying the averaging principle for stochastic differential equations with slow and fast time-scales, where the drift coefficients satisfy local Lipschitz conditions with respect to the slow and fast variables, and the  coefficients in the slow equation depend on time $t$ and $\omega$. Making use of the techniques of time discretization and truncation,  we prove that the slow component strongly converges to the solution of the corresponding averaged equation.
\end{abstract}

\date{\today}

\maketitle

\section{Introduction}

In this paper, we consider the following stochastic slow-fast system:
\begin{equation}\left\{\begin{array}{l}\label{Equation}
\displaystyle
d X^{\ep}_t = b(t, X^{\ep}_t, Y^{\ep}_t)dt+\sigma(t, X^{\ep}_t)d W^{1}_t,\quad X^{\ep}_0=x\in\RR^n, \\
d Y^{\ep}_t =\frac{1}{\ep}f(t,X^{\ep}_t, Y^{\ep}_t)dt+\frac{1}{\sqrt{\ep}}g(t,X^{\ep}_t, Y^{\ep}_t)d W^{2}_t,\quad Y^{\ep}_0=y\in\RR^m,
\end{array}\right.
\end{equation}
where $\ep$ is a small positive parameter describing the ratio of time scales between the slow component $X^{\ep}_t\in\RR^n$ and fast component $Y^{\ep}_t\in\RR^m$. Let $\{W^{1}_t\}_{t\geq 0}$ and $\{W^{2}_t\}_{t\geq 0}$ be mutually independent $d_1$ and $d_2$ dimensional standard Brownian motions on a complete probability space $(\Omega, \mathscr{F}, \mathbb{P})$ and $\{\mathscr{F}_{t},t\geq0\}$ is the natural filtration generated by $W^{1}_t$ and $W^{2}_t$. Let us consider the following given maps
\begin{eqnarray*}
&&b: [0, \infty)\times\RR^n\times\RR^m\times \Omega\rightarrow \RR^{n};\\
&& \sigma: [0, \infty)\times\RR^n\times \Omega\rightarrow \RR^{n\times d_1};\\
&&f: [0, \infty)\times\RR^n\times\RR^m\rightarrow \RR^{m};\\
&&g: [0, \infty)\times\RR^n\times\RR^m\rightarrow \RR^{m\times d_2}
\end{eqnarray*}
such that $b$, $\sigma$, $f$ and $g$ are continuous in $(x,y)\in \RR^{n}\times  \RR^{m}$ for each fixed $t\in [0, \infty)$, $\omega\in \Omega$, and progressively measurable, i.e., for each $t$, their restrictions to $[0, t]\times\Omega$ are $\mathcal{B}([0, t])\otimes\mathscr{F}_t$-measurable for any fixed $(x,y)\in \RR^{n}\times \RR^{m}$. In particular, for fixed $(x,y)\in \RR^{n}\times \RR^{m}$ and $t\in [0, \infty)$, $b(t,x,y)$ and $\sigma(t,x)$ are $\mathscr{F}_t$-measurable.

\vspace{1mm}
Under some reasonable assumptions, we intend to prove $X^{\ep}$ converges to $\bar{X}$ in the sense of $L^{p}(\Omega; C([0, T],\RR^n))$, $i.e.$, for some $p>0$,
\begin{eqnarray}
\lim_{\ep\rightarrow 0}\mathbb{E}\left(\sup_{t\in [0, T]}|X^{\ep}_t-\bar{X}_t|^p\right)=0,  \label{main result}
\end{eqnarray}
 where $\bar{X}$ is the solution of the corresponding averaged equation
\begin{equation}\left\{\begin{array}{l}
\displaystyle d\bar{X}_{t}=\bar{b}(t,\bar{X}_t)dt+\sigma(t, \bar{X}_t)d W^{1}_t.\\
\bar{X}_{0}=x,\end{array}\right. \label{1.3}
\end{equation}
Here $\bar{b}(t,x)=\int_{\RR^m}b(t,x,y)\mu^{t,x}(dy)$ and $\mu^{t,x}$ denotes the unique invariant measure for the transition semigroup  of the following frozen equation
\begin{equation}\left\{\begin{array}{l}\label{F}
\displaystyle
dY_{s}=f(t, x,Y_{s})ds+g(t, x,Y_{s})d\tilde{W}_{s}^{2},\\
Y_{0}=y,\\
\end{array}\right.
\end{equation}
where $\{\tilde {W}_{s}^{2}\}_{s\geq 0}$ is a $d_2$-dimensional standard Brownian motion on another complete probability space. Notice that for fixed $t\geq 0$ and $x\in\RR^n$, the solution of Eq. \eref{F}  is  a time-homogeneous Markov process, so its transition semigroup has a unique invariant measure $\mu^{t,x}$ depending on  $t$ and $x$ under appropriate conditions. Hence, the definition of the averaged coefficient $\bar b$ is meaningful.

\vspace{1mm}
Another simple understanding about the averaged coefficient is to change the time-dependent coefficients to time-independent coefficients. If $b$ and $\sigma$ are independent of $\omega$, then we define
$$
Z^{\ep}_t=\left(
                               \begin{array}{c}
                                 t \\
                                 X^{\ep}_t \\
                               \end{array}
                             \right)
,\quad \tilde{b}(z,y)=\left(\begin{array}{c}
                             1
                              \\
                               b(z,y) \\
                             \end{array}
                           \right)
$$
and
$$
\tilde{\sigma}(z)=\left(
                                  \begin{array}{cc}
                                   0 & 0 \\
                                    0 & \sigma(z,y) \\
                                  \end{array}
                                \right)
, \quad \tilde{W}^1_t=\left(
             \begin{array}{c}
               W_t \\
               W^1_t \\
             \end{array}
           \right).
$$
where $z\in \RR^{n+1}$, $\{W_t\}_{t\geq 0}$ is another one dimensional standard Brownian motion independent of $W^1_t$ and $W^2_t$. By an easy transformation,  the system  \eref{Equation} is then equivalent to   the following slow-fast system
\begin{equation}\left\{\begin{array}{l}\label{HSLS}
\displaystyle
d Z^{\ep}_t = \tilde b(Z^{\ep}_t, Y^{\ep}_t)dt+\tilde \sigma(Z^{\ep}_t)d \tilde W^{1}_t,\quad Z^{\ep}_0=\left(
                                                                                                          \begin{array}{c}
                                                                                                            0 \\
                                                                                                            x \\
                                                                                                          \end{array}
                                                                                                        \right)
, \\
d Y^{\ep}_t =\frac{1}{\ep}f(Z^{\ep}_t, Y^{\ep}_t)dt+\frac{1}{\sqrt{\ep}}g(Z^{\ep}_t, Y^{\ep}_t)d W^{2}_t,\quad Y^{\ep}_0=y,
\end{array}\right.
\end{equation}
where $Z^{\ep}_t\in\RR^{1+n}$ and $Y^{\ep}_t\in\RR^m$ are the corresponding slow and fast components for the new system \eref{HSLS} respectively. Notice that the system \eref{HSLS} is a time-independent case, and it is easy to see the corresponding frozen equation should be Eq. \eref{F}.

\vspace{1mm}
Although the coefficients depend on time in our paper,  it is different from the non-autonomous case in \cite{CL}. Recently, Cerrai \cite{CL} studied the averaging principle for non-autonomous slow-fast systems of stochastic reaction-diffusion equations, where the coefficients depend on time and satisfy the almost periodic in time condition. Because the corresponding frozen equation is a non-homogeneous Markov process and  the invariant measure does not exist any longer,  the assumption of almost periodic in time for the coefficients seems natural and it is used to define the averaged coefficient in a new way.

\vspace{1mm}
The theory of averaging principle has a long and rich history in multiscale problems, which arise from material sciences, chemistry, fluid dynamics, biology, climate dynamics and other application areas, see, e.g., \cite{BR,E,EL,HKW,PS,WTRY16} and references therein. The multiscale model is very common and involved by slow and fast components in mathematical models. For instance, dynamics of chemical reaction networks often take place on notably different times scales, from the order of nanoseconds ($\rm 10^{-9}$ s) to the order of several days. Studying the averaging principle is essential to describe the asymptotic behavior of slow component.

\vspace{1mm}
The averaging principle for stochastic differential equations (SDEs for short) was first studied by Khasminskii \cite{K1}, see, e.g. \cite{G,GJ,GL,L,XPG} (and the references therein) for further generalizations. However, most of the known results in the literature mainly considered the cases of coefficients satisfying Lipschitz continuous or sublinear growth conditions. It seems that there are few results about the non-Lipschitz case.  Veretennikov \cite{V0} established the averaging of systems of It\^{o} stochastic equations, where the drift coefficient $b$ is bounded and measurable $w.r.t.$ the slow variable and the other coefficients satisfy Lipschitz conditions. Then convergence in probability was obtained. Xu et al. \cite{XLM} proved the $L^2$ convergence for two-time-scales with special non-Lipschitz, but linear growth coefficients.

\vspace{1mm}
 However, in \cite{V0,XLM} it can not cover the superlinear growth case of drift coefficient $b$ such as $b(x,y)=x+y^3$. Hence, the motivation of this paper is to weaken the conditions on the drift coefficients $b$ and $f$ to local Lipschitz conditions $w.r.t.$ both the slow and fast variables, and to the case where the coefficients in the slow equation can depend on time $t$ and $\omega$,  which is inspired from the models in \cite[Chapter 3]{LR}.

\vspace{1mm}
Comparing with the known results, the main difficulties here are how to deal with the local Lispchitz continuity $w.r.t.$ the fast variable and the dependence on $\omega$ of the coefficients. In order to overcome these difficulties, we will continue to use the technique of stopping times very frequently. The main result is e.g. applicable to many slow-fast SDE models with polynomial drift coefficients. It is worth to mention that the approach based on time discretization will be used in the proof, so we need the local Lipschitz conditions instead of the one-sided type conditions in \cite[Theorem 3.1.1]{LR}.

\vspace{1mm}
The paper is organized as follows. In the next section, we introduce some notations
and assumptions that we use throughout the paper and formulate the main result. Section 3 is devoted to proving the strong convergence result. In Section 4, we will give some examples to illustrate the applicability of our result. The final section is the Appendix, where we present the detailed proof of existence and uniqueness of solutions for system \eref{Equation} and the corresponding averaged equation.

\vspace{1mm}
Please note that $C$ and $C_p$  denote some positive constants which may change from line to line throughout this paper, where $p$ is one or more than one parameter and $C_p$ is used to emphasize  that  the constant  depends on the corresponding parameter. $C_T$ will usually denote some nondecreasing function $w.r.t.$ $T$.

\section{Main results}\label{sec.prelim}

Now we impose the following assumptions on the coefficients $b,\sigma, f$ and $g$. Let $|\cdot|$ be the Euclidean norm, $\langle \cdot,\cdot\rangle$ be the Euclidean inner product and $\|\cdot\|$ be the matrix norm.

\vspace{0.3cm}
\smallskip
\noindent
\textbf{$({\bf H}_1)$}
(i) There exists $\theta_1\geq 0$ such that for any $t, R\geq 0$, $x_i\in\RR^n, y\in\RR^m$  with $|x_i|\leq R$,
\begin{eqnarray*}
2|b(t, x_1, y)-b(t, x_2, y)||x_1-x_2|+\|\sigma(t, x_1)-\sigma(t, x_2)\|^2\leq K_t(R)(1+|y|^{\theta_1})|x_1-x_2|^2,
\end{eqnarray*}
where $K_t(R)$ is an $\RR_{+}$-valued $\mathscr{F}_t$-adapted process satisfying  for all $R, T, p\in [0, \infty)$,
$$
\alpha_T(R):=\int^T_0 K_t(R)dt<\infty,\quad \text{on} \quad \Omega,
$$
$$
\EE e^{p\alpha_T(1)}<\infty, \quad \sup_{t\in [0, T]}\EE|K_t(1)|^4<\infty.
$$
Furthermore, there exists $R_0>0$, such that for any $R\geq R_0$, $T\geq 0$,
$$
\EE\int^T_0 [K_t(R)]^4 dt<\infty.
$$

(ii) There exist constants $\theta_2, \theta_3\geq 1$ and $\gamma_1\in(0, 1]$ such that for any $x\in\RR^n$, $y,y_1,y_2\in\RR^m$ and $T>0$ with $t,s\in[0, T]$,
\begin{eqnarray*}
|b(t, x, y_1)-b(t, x, y_2)|\leq C_T|y_1-y_2|\left[|y_1|^{\theta_2}+|y_2|^{\theta_2}+K_t(1)+|x|^{\theta_3}\right]
\end{eqnarray*}
and
\begin{eqnarray*}
|b(t, x, y)-b(s, x, y)|\leq C_T|t-s|^{\gamma_1}\left[|y|^{\theta_2}+|x|^{\theta_3}+Z_T\right], \quad \text{on} \quad \Omega,
\end{eqnarray*}
where $C_T>0$ and $Z_T$ is some random variable satisfying $\EE Z^2_T<\infty$.

(iii) There exist $\lambda_1\geq0$, $C>0$, $\theta_4\geq 2$ and $\theta_5, \theta_6\geq 1$ such that for any $t>0, x\in\RR^n, y\in\RR^m$,
\begin{eqnarray*}
2\langle x, b(t, x,y)\rangle\leq K_t(1)(1+|x|^{2})+\lambda_1|y|^{\theta_4}
\end{eqnarray*}
and
\begin{eqnarray*}
|b(t,x,y)|\leq K_t(1)+C(|x|^{\theta_5}+|y|^{\theta_6}),\quad \|\sigma(t,x)\|^2\leq K_t(1)+C|x|^2.
\end{eqnarray*}

 \smallskip
\noindent
\textbf{$({\bf H}_2)$}
(i) There exists $\beta>0$ such that for any $t\geq 0$, $x\in\RR^n, y_1,y_2\in\RR^m$,
\begin{eqnarray}
2\langle f(t,x, y_1)-f(t,x, y_2), y_1-y_2\rangle+\|g(t,x, y_1)-g(t,x, y_2)\|^2\leq -\beta|y_1-y_2|^2. \label{sm}
\end{eqnarray}

(ii) For any $T>0$, there exist $\gamma_2\in (0,1]$, $C_T>0$, $\alpha_i\geq 1$, $i=1,2,3,4$ such that for any $t, s\in [0,T]$ and $x_i\in\RR^n, y_i\in\RR^m$, $i=1,2$,
\begin{eqnarray*}
|f(t,x_1, y_1)-f(s,x_2, y_1)|\leq C_T (|t-s|^{\gamma_2}+|x_1-x_2|)(1+|x_1|^{\alpha_1}+|x_2|^{\alpha_1}+|y_1|^{\alpha_2});
\end{eqnarray*}
\begin{eqnarray*}
\|g(t, x_1, y_1)-g(s, x_2, y_2)\|\leq C_T(|t-s|^{\gamma_2}+|x_1-x_2|+|y_1-y_2|);
\end{eqnarray*}
\begin{eqnarray*}
 |f(t, x_1,y_1)|\leq C_T(1+|x_1|^{\alpha_3}+|y_1|^{\alpha_4});
\end{eqnarray*}
\begin{eqnarray*}
 \|g(t, x_1, y_1)\|\leq C_T(1+|x_1|+|y_1|).
\end{eqnarray*}

 \smallskip
\noindent
\textbf{$({\bf A}_{k})$} For some fixed $k\geq 2$ and any $T>0$, there exist $C_{T,k}, \beta_k>0$ such that for any $t\in [0, T]$, $x\in\RR^n, y\in\RR^m$,
\begin{eqnarray}
2\langle y, f(t, x, y)\rangle+(k-1)\|g(t,x,y)\|^2\leq -\beta_k|y|^2-\lambda_2|y|^{\theta_4}+C_{T,k} (|x|^{\frac{4}{\theta_4}}+1),\label{C2.2}
\end{eqnarray}
where $\lambda_2=0$ if $\lambda_1=0$, and $\lambda_2>0$ otherwise.

\begin{remark}\label{R2.1} 
 (1) Condition \eref{sm} is called strict monotonicity condition, which ensures that exponential ergodicity holds (see Proposition \ref{Rem 4.1} below). Condition \eref{C2.2} is called strict coercivity condition, which is used to guarantee the existence of invariant measures for the frozen equation (see Eq. (\ref{FEQ1}) below). Hence the uniqueness of invariant measures for the frozen equation follows (see Proposition \ref{invariant} below).

(2) The powers $\theta_4$ and $\frac{4}{\theta_4}$ in \eref{C2.2} are used to ensure the existence and uniqueness of solutions to the system \eref{Equation}  and the corresponding averaged equation (see Eq. (\ref{3.1}) below) respectively.

(3) If $k_1>k_2\geq 2$, then \textbf{$({\bf A}_{k_1})$} implies \textbf{$({\bf A}_{k_2})$}.

(4) We will give some  examples in Section 4 to show  the assumptions above hold for many drift coefficients of polynomial type.
\end{remark}

The following theorem is the existence and uniqueness result for system \eref{Equation}, which can be obtained  using the classical result due to Krylov (cf. \cite[Chapter 3]{LR}). The detailed proof will be given in the Appendix.
\begin{theorem}\label{main}
Suppose that \textbf{$({\bf H}_1)$}, \textbf{$({\bf H}_2)$} and \textbf{$({\bf A}_2)$}hold. Let $\ep_0=\frac{\lambda_2}{\lambda_1}$ if $\lambda_1>0$, and $\ep_0=1$ otherwise. Then for any $\ep\in (0, \ep_0)$, any given initial values $x\in\RR^n, y\in \RR^m$, there exists a unique solution $\{(X^{\ep}_t,Y^{\ep}_t), t\geq 0\}$ to system \eref{Equation} and for all $T>0$,
$(X^{\ep},Y^{\ep})\in C([0,T]; \RR^n)\times C([0,T]; \RR^m), \PP-a.s.$ and for all $t\in [0, T]$,
\begin{equation}\left\{\begin{array}{l}\label{mild solution}
\displaystyle
X^{\ep}_t=x+\int^t_0b(s, X^{\ep}_s, Y^{\ep}_s)ds+\int^t_0\sigma(s, X^{\ep}_s)d W^{1}_s,\\
\displaystyle
Y^{\ep}_t=y+\frac{1}{\ep} \int^t_0f(s,X^{\ep}_s,Y^{\ep}_s)ds
+\frac{1}{\sqrt{\ep}} \int^t_0 g(s, X^{\ep}_s, Y^{\ep}_s)dW^{2}_s.
\end{array}\right.
\end{equation}
\end{theorem}

Now we formulate the main result of this work.
\begin{theorem}\label{main result 1}
Suppose that \textbf{$({\bf H}_1)$}  and \textbf{$({\bf H}_2)$}  hold.

(i) If $\lambda_1=0$ in \textbf{$({\bf H}_1)$} and \textbf{$({\bf A}_{\tilde \theta_1})$} holds for $\tilde \theta_1=\max\{4\theta_1, 2\theta_2+2, 2\theta_6, 4\alpha_2\}$. Then for any $p>0$  we have
\begin{eqnarray}
\lim_{\ep\rightarrow 0}\mathbb{E}\left(\sup_{t\in [0, T]}|X_{t}^{\ep}-\bar{X}_{t}|^p\right)=0. \label{R1}
\end{eqnarray}

(ii) If $\lambda_1>0$ in \textbf{$({\bf H}_1)$} and \textbf{$({\bf A}_{k})$} holds for some $k>\tilde \theta_2$ with $\tilde \theta_2=\max\{4\theta_1, 2\theta_2+2, 2\theta_6, 4\alpha_2, \newline \theta_5\theta_4, 2\alpha_1\theta_4\}$. Then for any $0<p<\frac{2k}{\theta_4}$ we have
\begin{eqnarray}
\lim_{\ep\rightarrow 0}\mathbb{E}\left(\sup_{t\in [0, T]}|X_{t}^{\ep}-\bar{X}_{t}|^{p}\right)=0. \label{R2}
\end{eqnarray}
Here $\bar{X}$ is the solution of the following averaged equation
\begin{equation}\left\{\begin{array}{l}
\displaystyle d\bar{X}_{t}=\bar{b}(t,\bar{X}_t)dt+\sigma(t,\bar{X}_t)d W^{1}_t,\\
\bar{X}_{0}=x,\end{array}\right. \label{1.3}
\end{equation}
where $\bar{b}(t,x)=\int_{\RR^m}b(t,x,y)\mu^{t,x}(dy)$ and $\mu^{t,x}$ denotes the unique invariant measure for the transition semigroup of the corresponding frozen equation
\begin{equation}\left\{\begin{array}{l}\label{FEQ1}
\displaystyle
dY_{s}=f(t,x,Y_{s})ds+g(t,x,Y_{s})d\tilde {W}_{s}^{2},\\
Y_{0}=y,\\
\end{array}\right.
\end{equation}
where $\{\tilde{W}_{s}^{2}\}_{s\geq 0}$ is a $d_2$-dimensional Brownian motion on another complete probability space $(\tilde{\Omega}, \tilde{\mathscr{F}}, \tilde{\mathbb{P}})$.
\end{theorem}

\section{Proof of the Main Result}

This section is devoted to proving Theorem \ref{main result 1}.
The proof consists of the following steps.
Firstly, we give some  a-priori estimates for the solution $(X^{\ep}_t, Y^{\ep}_t)$
to the system (\ref{Equation}).
Secondly, following the  discretization techniques inspired by Khasminskii in \cite{K1},
we introduce an auxiliary process $(\hat{X}_{t}^{\ep},\hat{Y}_{t}^{\ep})$ for which we derive uniform bounds. Making use of the  stopping time techniques inspired by \cite{DSXZ}, we control the (difference) process $X^{\ep}_t-\hat{X}_{t}^{\ep}$ before the stopping time.
Thirdly, based on the ergodic property of the frozen equation, we obtain appropriate control of $\hat{X}_{t}^{\ep}-\bar{X}_{t}$ before the stopping time.
Finally, we shall use the a-priori estimates of the solution to control the difference after the stopping time. Note that we always assume that \textbf{$({\bf H}_1)$} and \textbf{$({\bf H}_2)$} hold and from now on we fix some initial values $x\in \RR^{n}, y\in \RR^{m}$  in this section.

\vskip 0.2cm

 \subsection{Some a-priori estimates of $(X^{\ep}_t, Y^{\ep}_t)$}

In this subsection, we prove some uniform bounds $w.r.t.$ $\ep \in (0,\ep_0)$ for the moments of the solution ($X_{t}^{\ep}, Y_{t}^{\ep})$ to system (\ref{Equation}).
\begin{lemma} \label{PMY} 
(i) If $\lambda_1=0$ in \textbf{$({\bf H}_1)$} and \textbf{$({\bf A}_{k\theta_4})$} holds for some $k\geq \frac{2}{\theta_4}$, then for any $T, p>0$, there exist  positive constants $C_{T, p}, C_{T,k}$ such that
\begin{eqnarray*}
\sup_{\ep\in(0,\ep_0)}\mathbb{E}\left(\sup_{t\in [0, T]}|X_{t}^{\ep}|^{p}\right)\leq C_{T, p}(1+|x|^{p})\label{X}
\end{eqnarray*}
and
\begin{eqnarray*}
\sup_{\ep\in(0,\ep_0)}\sup_{t\in [0, T]}\mathbb{E}|Y_{t}^{\ep}|^{k\theta_4}\leq C_{T,k}(1+|x|^{2k}+|y|^{k\theta_4}).\label{Y}
\end{eqnarray*}

(ii) If $\lambda_1>0$ in \textbf{$({\bf H}_1)$} and \textbf{$({\bf A}_{k\theta_4})$} holds for some $k\geq 1$, then for any $T>0$, $k'<k$, there exists a positive constant $C_{T,k}$ such that
\begin{eqnarray*}
\sup_{\ep\in(0,\ep_0)}\mathbb{E}\left(\sup_{t\in [0, T]}|X_{t}^{\ep}|^{2k'}\right)\leq C_{T,k}(|x|^{2k'}+|y|^{k'\theta_4}+1)\label{X1}
\end{eqnarray*}
and
\begin{eqnarray*}
\sup_{\ep\in(0,\ep_0)}\sup_{t\in [0, T]}\mathbb{E}|Y_{t}^{\ep}|^{k'\theta_4}\leq C_{T,k}(|x|^{2k'}+|y|^{k'\theta_4}+1).\label{Y2}
\end{eqnarray*}
\end{lemma}
\begin{proof}
$(i)$  According to It\^{o}'s formula and \textbf{$({\bf H}_1)$} with $\lambda_1=0$,  we have for any $p\geq 2$,
\begin{eqnarray*}
&&e^{-\frac{p}{2}\alpha_t(1)}|X_{t}^{\ep}|^{p}\\
=\!\!\!\!\!\!\!\!&&|x|^{p}+\frac{p}{2}\int^t_0 e^{-\frac{p}{2}\alpha_s(1)}\left[-K_s(1)\right]|X_{s}^{\ep}|^p ds+p\int_{0} ^{t}e^{-\frac{p}{2}\alpha_s(1)}|X_{s}^{\ep}|^{p-2}\langle X_{s}^{\ep}, b(s,X_{s}^{\ep}, Y_{s}^{\ep})\rangle ds\\
&&+\frac{p}{2}\int_{0} ^{t}e^{-\frac{p}{2}\alpha_s(1)}|X_{s}^{\ep}|^{p-2}\|\sigma(s,X_{s}^{\ep})\|^2ds+\frac{p(p-2)}{2}\int_{0} ^{t}e^{-\frac{p}{2}\alpha_s(1)}|X_{s}^{\ep}|^{p-4}\left|\langle X_{s}^{\ep}, \sigma(s,X_{s}^{\ep})\rangle\right|^2 ds\\
&&
+p\int_{0} ^{t}e^{-\frac{p}{2}\alpha_s(1)}|X_{s}^{\ep}|^{p-2}\langle X_{s}^{\ep}, \sigma(s,X_{s}^{\ep})dW^{1}_s\rangle\\
\leq\!\!\!\!\!\!\!\!&&|x|^{p}+\frac{p}{2}\int^t_0 e^{-\frac{p}{2}\alpha_s(1)}\left[-K_s(1)\right]|X_{s}^{\ep}|^p ds+\frac{p}{2}\int_{0} ^{t}e^{-\frac{p}{2}\alpha_s(1)}|X_{s}^{\ep}|^{p-2}K_s(1)(1+|X^{\ep}_s|^{2})ds\\
&&+\frac{p(p-1)}{2}\int_{0} ^{t}e^{-\frac{p}{2}\alpha_s(1)}|X_{s}^{\ep}|^{p-2}[K_s(1)+C|X^{\ep}_s|^2]ds\\
&&+p\int_{0} ^{t}e^{-\frac{p}{2}\alpha_s(1)}|X_{s}^{\ep}|^{p-2}\langle X_{s}^{\ep}, \sigma(s,X_{s}^{\ep})dW^{1}_s\rangle\\
\leq\!\!\!\!\!\!\!\!&&|x|^{p}+\frac{p^2}{2}\int^t_0 e^{-\frac{p}{2}\alpha_s(1)}|X_{s}^{\ep}|^{p-2}K_s(1)ds+C_p\int_{0} ^{t}e^{-\frac{p}{2}\alpha_s(1)}|X^{\ep}_s|^p ds\\
&&+p\int_{0} ^{t}e^{-\frac{p}{2}\alpha_s(1)}|X_{s}^{\ep}|^{p-2}\langle X_{s}^{\ep}, \sigma(s,X_{s}^{\ep})dW^{1}_s\rangle.
\end{eqnarray*}
Then by the Burkholder-Davis-Gundy inequality and Young's inequality, for any $T>0$, we have
\begin{eqnarray*}
&&\EE\left[\sup_{t\in [0,T]}\left(e^{-\frac{p}{2}\alpha_t(1)}|X_{t}^{\ep}|^{p}\right)\right]\\
\leq\!\!\!\!\!\!\!\!&&|x|^{p}+\frac{p^2}{2}\EE\int^T_0 e^{-\frac{p}{2}\alpha_t(1)}|X_{t}^{\ep}|^{p-2}K_t(1)dt+C_{p}\int^T_0 \EE \left(e^{-\frac{p}{2}\alpha_t(1)}|X_{t}^{\ep}|^{p}\right) dt\\
&&+C_p\EE\left[\int^T_0 e^{-p\alpha_t(1)}|X_{t}^{\ep}|^{2p-2}(K_t(1)+C|X^{\ep}_t|^2)dt\right]^{1/2}\\
\leq\!\!\!\!\!\!\!\!&&|x|^{p}+\frac{1}{4}\EE\left[\sup_{t\in [0,T]}\left(e^{-\frac{p}{2}\alpha_t(1)}|X_{t}^{\ep}|^{p}\right)\right]+C_p\EE\left[\int^T_0 e^{-\alpha_s(1)}K_s(1)ds\right]^{p/2}\\
&&+C_{p}\int^T_0 \EE \left(e^{-\frac{p}{2}\alpha_t(1)}|X_{t}^{\ep}|^{p}\right) dt+\frac{1}{4}\EE\left[\sup_{t\in [0, T]}\left(e^{-\frac{p}{2}\alpha_t(1)}|X_{t}^{\ep}|^{p}\right)\right]\\
&&+C_{p}\EE\left[\int^T_0 e^{-\alpha_t(1)}(K_t(1)+C|X^{\ep}_t|^2)dt\right]^{\frac{p}{2}},
\end{eqnarray*}
which implies that
\begin{eqnarray*}
\EE\left[\sup_{t\in [0,T]}\left(e^{-\frac{p}{2}\alpha_t(1)}|X_{t}^{\ep}|^{p}\right)\right]\leq\!\!\!\!\!\!\!\!&&C_{p}(|x|^{p}+1)+C_{T,p}\int^T_0 \EE \left(e^{-\frac{p}{2}\alpha_t(1)}|X_{t}^{\ep}|^{p}\right) dt.
\end{eqnarray*}
Then Gronwall's inequality yields that
\begin{eqnarray*}
\EE\left[\sup_{t\in [0,T]}\left(e^{-\frac{p}{2}\alpha_t(1)}|X_{t}^{\ep}|^{p}\right)\right]\leq C_{T,p}(|x|^{p}+1).
\end{eqnarray*}
Hence, by H\"{o}lder inequality and since $\EE e^{p\alpha_T(1)}<\infty$ for any $p>0$ , we obtain for any $p>0$
\begin{eqnarray*}
\EE\left(\sup_{t\in [0,T]}|X_{t}^{\ep}|^{p}\right)\leq\!\!\!\!\!\!\!\!&&\left\{\EE\left[\sup_{t\in [0,T]}\left(e^{-p\alpha_t(1)}|X_{t}^{\ep}|^{2p}\right)\right]\right\}^{\frac{1}{2}}\cdot \left[\EE e^{p\alpha_T(1)}\right]^{\frac{1}{2}}\\
\leq\!\!\!\!\!\!\!\!&&C_{T,p}(|x|^{p}+1).
\end{eqnarray*}

By It\^{o}'s formula we also have
\begin{eqnarray*}
\mathbb{E}|Y_{t}^{\ep}|^{k\theta_4}=\!\!\!\!\!\!\!\!&&\frac{k\theta_4}{\ep}\int^t_0\mathbb{E}\left[| Y_{s}^{\ep}|^{k\theta_4-2}\langle f(s,X_{s}^{\ep},Y_{s}^{\ep}),Y_{s}^{\ep}\rangle\right] ds+\frac{k\theta_4}{2\ep}\int^t_0 \mathbb{E}\left[|Y_{s}^{\ep}|^{k\theta_4-2}\|g(s,X_{s}^{\ep},Y_{s}^{\ep})\|^2\right] ds\nonumber \\
 \!\!\!\!\!\!\!\!&& +\frac{k\theta_4(k\theta_4-2)}{2\ep}\int^t_0\mathbb{E}\left[|Y_{s}^{\ep}|^{k\theta_4-4}\cdot|\langle Y^\ep_s, g(s,X_{s}^{\ep},Y_{s}^{\ep})\rangle|^2\right] ds.
\end{eqnarray*}
If \textbf{$({\bf A}_{k\theta_4})$} holds for $k\geq \frac{2}{\theta_4}$, then there exist $\tilde\beta_k, C_{T,k}>0$ such that
\begin{eqnarray*}
\frac{d}{dt}\mathbb{E}|Y_{t}^{\ep}|^{k\theta_4}\leq\!\!\!\!\!\!\!\!&&\frac{k\theta_4}{2\ep}\EE\left[| Y_{t}^{\ep}|^{k\theta_4-2}\left(2\langle f(t, X_{t}^{\ep},Y_{t}^{\ep}),Y_{t}^{\ep}\rangle+(k\theta_4-1)\|g(t, X_{t}^{\ep},Y_{t}^{\ep})\|^2\right)\right]\\
 \leq\!\!\!\!\!\!\!\!&&-\frac{\tilde \beta_k}{\ep}\mathbb{E}|Y_{t}^{\ep}|^{k\theta_4}
+\frac{C_{T,k}}{\ep}\left(\EE|X_{t}^{\ep}|^{2k}+1\right).
\end{eqnarray*}
Hence, by the comparison theorem we obtain
\begin{eqnarray*}
\mathbb{E}|Y_{t}^{\ep}|^{k\theta_4}\leq\!\!\!\!\!\!\!\!&&|y|^{k\theta_4}e^{-\frac{\tilde \beta_k}{\ep}t}+\frac{C_{T,k}}{\ep}\int^t_0 e^{-\frac{\tilde \beta_k}{\ep}(t-s)}\Big(1+\EE|X_{s}^{\ep}|^{2k}\Big)ds\nonumber\\
\leq\!\!\!\!\!\!\!\!&&C_{T,k}(1+|x|^{2k}+|y|^{k\theta_4}),\label{4.4.3}
\end{eqnarray*}
which gives the statement $(i)$.

\vspace{0.3cm}
$(ii)$. Notice that since \textbf{$({\bf H}_1)$} holds with $\lambda_1>0$, for any $k\geq 1$,  It\^{o}'s formula implies that
\begin{eqnarray*}
&&e^{-k\alpha_t(1)}|X_{t}^{\ep}|^{2k}\\
=\!\!\!\!\!\!\!\!&&|x|^{2k}+k\int^t_0 e^{-k\alpha_s(1)}\left[-K_s(1)\right]|X_{s}^{\ep}|^{2k} ds+2k\int_{0} ^{t}e^{-k\alpha_s(1)}|X_{s}^{\ep}|^{2k-2}\langle X_{s}^{\ep}, b(s,X_{s}^{\ep}, Y_{s}^{\ep})\rangle ds\\
&&+k\int_{0} ^{t}e^{-k\alpha_s(1)}|X_{s}^{\ep}|^{2k-2}\|\sigma(s,X_{s}^{\ep})\|^2ds+2k(k-1)\int_{0} ^{t}e^{-k\alpha_s(1)}|X_{s}^{\ep}|^{2k-4}\left|\langle X_{s}^{\ep}, \sigma(s,X_{s}^{\ep})\rangle\right|^2 ds\\
&&
+2k\int_{0} ^{t}e^{-k\alpha_s(1)}|X_{s}^{\ep}|^{2k-2}\langle X_{s}^{\ep}, \sigma(s,X_{s}^{\ep})dW^{1}_s\rangle\\
\leq\!\!\!\!\!\!\!\!&&|x|^{2k}+k\int^t_0 e^{-k\alpha_s(1)}\left[-K_s(1)\right]|X_{s}^{\ep}|^{2k} ds+k\int_{0} ^{t}e^{-k\alpha_s(1)}|X_{s}^{\ep}|^{2k-2}K_s(1)(1+|X^{\ep}_s|^{2})ds\\
&&+k\lambda_1\int_{0} ^{t}e^{-k\alpha_s(1)}|X_{s}^{\ep}|^{2k-2}|Y^{\ep}_s|^{\theta_4}ds+k(2k-1)\int_{0} ^{t}e^{-k\alpha_s(1)}|X_{s}^{\ep}|^{2k-2}[K_s(1)+C|X^{\ep}_s|^2]ds\\
&&+2k\int_{0} ^{t}e^{-k\alpha_s(1)}|X_{s}^{\ep}|^{2k-2}\langle X_{s}^{\ep}, \sigma(s,X_{s}^{\ep})dW^{1}_s\rangle\\
\leq\!\!\!\!\!\!\!\!&&|x|^{2k}+2 k^2 \int^t_0 e^{-k\alpha_s(1)}|X_{s}^{\ep}|^{2k-2}K_s(1)ds+C_k\int_{0} ^{t}e^{-k\alpha_s(1)}|Y^{\ep}_s|^{k\theta_4} ds\\
&&+C_k\int_{0} ^{t}e^{-k\alpha_s(1)}|X^{\ep}_s|^{2k} ds+2k\int_{0} ^{t}e^{-k\alpha_s(1)}|X_{s}^{\ep}|^{2k-2}\langle X_{s}^{\ep}, \sigma(s,X_{s}^{\ep})dW^{1}_s\rangle.
\end{eqnarray*}
Then by the Burkholder-Davis-Gundy inequality and Young's inequality,  we have
\begin{eqnarray*}
&&\EE\left[\sup_{t\in [0,T]}\left(e^{-k\alpha_t(1)}|X_{t}^{\ep}|^{2k}\right)\right]\\
\leq\!\!\!\!\!\!\!\!&&|x|^{2k}+2 k^2\int^T_0 \EE\left[e^{-\frac{k}{2}\alpha_t(1)}|X_{t}^{\ep}|^{2k-2}K_t(1)\right]dt+C_k\int_{0} ^{T}\EE\left(e^{-k\alpha_t(1)}|Y^{\ep}_t|^{k\theta_4}\right) dt\\
&&+C_{k}\int^T_0 \EE \left(e^{-k\alpha_t(1)}|X_{t}^{\ep}|^{2k}\right) dt+C_k\EE\left[\int^T_0 e^{-2k\alpha_t(1)}|X_{t}^{\ep}|^{4k-2}(K_t(1)+C|X^{\ep}_t|^2)dt\right]^{1/2}\\
\leq\!\!\!\!\!\!\!\!&&|x|^{2k}+\frac{1}{4}\EE\left[\sup_{t\in [0,T]}\left(e^{-k\alpha_t(1)}|X_{t}^{\ep}|^{2k}\right)\right]+C_k\EE\left[\int^T_0 e^{-\alpha_s(1)}K_s(1)ds\right]^{k}\\
&&+C_k\int_{0} ^{T}\EE\left(e^{-k\alpha_t(1)}|Y^{\ep}_t|^{k\theta_4}\right) dt+C_{k}\int^T_0 \EE \left(e^{-k\alpha_t(1)}|X_{t}^{\ep}|^{2k}\right) dt\\
&&+\frac{1}{4}\EE\left[\sup_{t\in [0, T]}\left(e^{-k\alpha_t(1)}|X_{t}^{\ep}|^{2k}\right)\right]+C_{k}\EE\left[\int^T_0 e^{-\alpha_t(1)}(K_t(1)+C|X^{\ep}_t|^2)dt\right]^{k},
\end{eqnarray*}
which implies that
\begin{eqnarray}
\EE\left[\sup_{t\in [0,T]}\left(e^{-k\alpha_t(1)}|X_{t}^{\ep}|^{2k}\right)\right]\leq\!\!\!\!\!\!\!\!&&C_{T,k}(|x|^{2k}+1)+C_{k}\int^T_0 \EE \left(e^{-k\alpha_t(1)}|Y_{t}^{\ep}|^{k\theta_4} \right)dt\nonumber\\
&&+C_{T,k}\int^T_0 \EE \left(e^{-k\alpha_t(1)}|X_{t}^{\ep}|^{2k}\right) dt.\label{I3.8}
\end{eqnarray}

Using It\^{o}'s formula again we have
\begin{eqnarray*}
&&\mathbb{E}\left(e^{-k\alpha_t(1)}|Y_{t}^{\ep}|^{k\theta_4}\right)\\
=\!\!\!\!\!\!\!\!&&\int^t_0\EE\left[e^{-k\alpha_s(1)}[-k K_s(1)]|Y^{\ep}_s|^{k\theta_4}\right]ds+\frac{k\theta_4}{\ep}\int^t_0\mathbb{E}\left[e^{-k\alpha_s(1)}| Y_{s}^{\ep}|^{k\theta_4-2}\langle f(s, X_{s}^{\ep},Y_{s}^{\ep}),Y_{s}^{\ep}\rangle \right]ds\\
&&+\frac{k\theta_4}{2\ep}\int^t_0\mathbb{E}\left[e^{-k\alpha_s(1)}|Y_{s}^{\ep}|^{k\theta_4-2}\|g(s, X_{s}^{\ep},Y_{s}^{\ep})\|^2\right]ds\\
&&+\frac{k\theta_4(k\theta_4-2)}{2\ep}\int^t_0\mathbb{E}\left[e^{-k\alpha_s(1)}|Y_{s}^{\ep}|^{k\theta_4-4}\cdot|\langle Y^\ep_t, g(s, X_{s}^{\ep},Y_{s}^{\ep})\rangle|^2\right]ds.
\end{eqnarray*}
For any $t\in [0, T]$,  \textbf{$({\bf A}_{k\theta_4})$} implies that there exists $\tilde{\beta}_k, C_{T,k}>0$ such that
\begin{eqnarray*}
\frac{d}{dt}\mathbb{E}\left(e^{-k\alpha_t(1)}|Y_{t}^{\ep}|^{k\theta_4}\right)\leq\!\!\!\!\!\!\!\!&&\frac{k\theta_4}{2\ep}\EE\left[e^{-k\alpha_s(1)}| Y_{t}^{\ep}|^{k\theta_4-2}\left(2\langle f(t,X_{t}^{\ep},Y_{t}^{\ep}),Y_{t}^{\ep}\rangle+(k\theta_4-1)\|g(t,X_{t}^{\ep},Y_{t}^{\ep})\|^2\right)\right]\\
 \leq\!\!\!\!\!\!\!\!&&-\frac{\tilde{\beta}_k}{\ep}\mathbb{E}\left(e^{-k\alpha_t(1)}|Y_{t}^{\ep}|^{k\theta_4}\right)
+\frac{C_{T,k}}{\ep}\left[\EE\left(e^{-k\alpha_t(1)}|X_{t}^{\ep}|^{2k}\right)+1\right].
\end{eqnarray*}
By the comparison theorem  there exists $\tilde{\beta}_k>0$
\begin{eqnarray}
\mathbb{E}\left(e^{-k\alpha_t(1)}|Y_{t}^{\ep}|^{k\theta_4}\right)\leq\!\!\!\!\!\!\!\!&&|y|^{k\theta_4}e^{-\frac{\tilde{\beta}_k t}{\ep}}+\frac{C_{T,k}}{\ep}\int^t_0 e^{-\frac{\tilde {\beta}_k (t-s)}{\ep}}\left(\EE\left[e^{-k\alpha_s(1)}|X_{s}^{\ep}|^{2k}\right]+1\right)ds\nonumber\\
\leq\!\!\!\!\!\!\!\!&&|y|^{k\theta_4}+C_{T,k}\EE\left[\sup_{s\in[0,t]}\left(e^{-k\alpha_s(1)}|X_{s}^{\ep}|^{2k}\right)\right]+C_{T,k}.\label{I3.9}
\end{eqnarray}
Combining this with \eref{I3.8} we obtain
\begin{eqnarray*}
\EE\left[\sup_{t\in [0,T]}\left(e^{-k\alpha_t(1)}|X_{t}^{\ep}|^{2k}\right)\right]\leq C_{T,k}(|x|^{2k}+|y|^{k\theta_4}+1)+C_{T,k}\int^T_0 \EE\left[\sup_{s\in[0, t]}\left(e^{-k\alpha_s(1)}|X_{s}^{\ep}|^{2k}\right)\right] dt.
\end{eqnarray*}
Hence Gronwall's inequality implies that
\begin{eqnarray}
\EE\left[\sup_{t\in [0,T]}\left(e^{-k\alpha_t(1)}|X_{t}^{\ep}|^{2k}\right)\right]\leq C_{T,k}(|x|^{2k}+|y|^{k\theta_4}+1).\label{I3.10}
\end{eqnarray}
By \eref{I3.9} and \eref{I3.10}, we also have
\begin{eqnarray*}
\sup_{t\in[0, T]}\mathbb{E}\left(e^{-k\alpha_t(1)}|Y_{t}^{\ep}|^{k\theta_4}\right)\leq C_{T,k}(|x|^{2k}+|y|^{k\theta_4}+1).
\end{eqnarray*}
Hence, by H\"{o}lder's inequality and since $\EE e^{p\alpha_T(1)}<\infty$ for any $p>0$ , we obtain for any $k'<k$,
\begin{eqnarray*}
\EE\left(\sup_{t\in [0,T]}|X_{t}^{\ep}|^{2k'}\right)\leq\!\!\!\!\!\!\!\!&&\left\{\EE\left[\sup_{t\in [0,T]}\left(e^{-k\alpha_t(1)}|X_{t}^{\ep}|^{2k}\right)\right]\right\}^{\frac{k'}{k}}\cdot \left[\EE e^{\frac{kk'}{k-k'}\alpha_T(1)}\right]^{\frac{k-k'}{k}}\\
\leq\!\!\!\!\!\!\!\!&&C_{T,k}(|x|^{2k'}+|y|^{k'\theta_4}+1).
\end{eqnarray*}
Similarly, we have
\begin{eqnarray*}
\sup_{t\in[0, T]}\mathbb{E}|Y_{t}^{\ep}|^{k'\theta_4}\leq C_{T,k}(|x|^{2k'}+|y|^{k'\theta_4 }+1).
\end{eqnarray*}
 Hence the proof is complete.
\end{proof}

\begin{lemma} \label{COX} Assume that either \textbf{$({\bf H}_1)$} with $\lambda_1=0$ and \textbf{$({\bf A}_{2\theta_6})$} hold or \textbf{$({\bf H}_1)$} with $\lambda_1>0$ and \textbf{$({\bf A}_{k})$} with some $k>\max\{2\theta_6, \theta_5\theta_4\}$ hold. Then for any $T>0$, $0\leq t\leq t+h\leq T$, there exists  $C_{T,x,y}>0$ such that
\begin{align*}
\sup_{\ep\in(0,1)}\mathbb{E}|X_{t+h}^{\ep}-X_{t}^{\ep}|^{2}\leq C_{T,x,y}h.
\end{align*}
\end{lemma}

\begin{proof}
We have
\begin{eqnarray}
X_{t+h}^{\ep}-X_{t}^{\ep}=\int^{t+h}_t b(s, X^{\ep}_s, Y^{\ep}_s)ds+\int^{t+h}_t\sigma(s, X^{\ep}_s)dW^1_s.\nonumber
\end{eqnarray}
Then by condition \textbf{$({\bf H}_1)$} and Lemma \ref{PMY}, we get
\begin{eqnarray*}
\EE|X_{t+h}^{\ep}-X_{t}^{\ep}|^2\leq &&\!\!\!\!\!\!\!\!C\EE\left|\int^{t+h}_t b(s, X^{\ep}_s, Y^{\ep}_s)ds\right|^2+C\EE\left|\int^{t+h}_t\sigma(s, X^{\ep}_s)dW^1_s\right|^2\\
\leq&& \!\!\!\!\!\!\!\!C \EE \left|\int^{t+h}_t|b(s, X^{\ep}_s, Y^{\ep}_s)| ds\right|^2+C\int^{t+h}_t \EE\|\sigma(s, X^{\ep}_s)\|^2ds\\
\leq &&\!\!\!\!\!\!\!\! C h\EE\int^{t+h}_t \left[(K_s(1))^2+|X^{\ep}_s|^{2\theta_5}+| Y^{\ep}_s|^{2\theta_6} \right]ds+C\int^{t+h}_t \EE(K_s(1)+C|X^{\ep}_s|^2)ds\\
\leq &&\!\!\!\!\!\!\!\! C_{T,x,y}h.
\end{eqnarray*}
The proof is complete.
\end{proof}

\vskip 0.2cm

\subsection{ Estimates of the auxiliary process $(\hat{X}_{t}^{\ep},\hat{Y}_{t}^{\ep})$}

Following the idea inspired by Khasminskii in \cite{K1},
we introduce an auxiliary process $(\hat{X}_{t}^{\ep},\hat{Y}_{t}^{\ep})\in \RR^n\times \RR^m$
and divide $[0,T]$ into intervals depending of size $\delta$, where $\delta$ is a fixed positive number depending on $\ep$, which will be chosen later.
We construct a process $\hat{Y}_{t}^{\ep}$ with initial value $\hat{Y}_{0}^{\ep}=Y^{\ep}_{0}=y$,
and for $t\in[k\delta,\min((k+1)\delta,T)]$,
\begin{eqnarray*}
\hat{Y}_{t}^{\ep}=\hat Y_{k\delta}^{\ep}+\frac{1}{\ep}\int_{k\delta}^{t}
f(k\delta, X_{k\delta}^{\ep},\hat{Y}_{s}^{\ep})ds+\frac{1}{\sqrt{\ep}}\int_{k\delta}^{t}g(k\delta, X_{k\delta}^{\ep},\hat{Y}_{s}^{\ep})dW^{2}_s,
\end{eqnarray*}
i.e.,
\begin{eqnarray*}
\hat{Y}_{t}^{\ep}=y+\frac{1}{\ep}\int_{0}^{t} f(s(\delta), X_{s(\delta)}^{\ep},\hat{Y}_{s}^{\ep})ds+\frac{1}{\sqrt{\ep}}\int_{0}^{t}g(s(\delta), X_{s(\delta)}^{\ep},\hat{Y}_{s}^{\ep})d W^{2}_s,
\end{eqnarray*}
where $s(\delta)=[{s}/{\delta}]\delta$ and $[{s}/{\delta}]$ is the integer part of ${s}/{\delta}$. Also, we define the process $\hat{X}_{t}^{\ep}$ by
$$
\hat{X}_{t}^{\ep}=x+\int_{0}^{t}b(s(\delta), X_{s(\delta)}^{\ep},\hat{Y}_{s}^{\ep})ds+\int_{0}^{t}\sigma(s, X^{\ep}_s)dW^{1}_s.
$$

By the construction of $\hat{Y}_{t}^{\ep}$ and similar arguments as in Lemma \ref{PMY}, it is easy to obtain the following estimate whose proof we omit here.
\begin{lemma} \label{MDY}
(i) If $\lambda_1=0$ in \textbf{$({\bf H}_1)$} and \textbf{$({\bf A}_{k\theta_4})$} holds with some $k\geq \frac{2}{\theta_4}$, then for any $T>0$, there exists a constant $C_{T,k}>0$ such that
\begin{eqnarray*}
\sup_{\ep\in(0,\ep_0)}\sup_{t\in [0, T]}\mathbb{E}|\hat Y_{t}^{\ep}|^{k\theta_4}\leq C_{T,k}(|x|^{2k}+|y|^{k\theta_4}+1).\label{Y}
\end{eqnarray*}

(ii) If $\lambda_1>0$ in \textbf{$({\bf H}_1)$} and \textbf{$({\bf A}_{k\theta_4})$} holds with some $k\geq 1$, then for any $T>0$, $k'<k$, there exists a constant $C_{T,k}>0$ such that
\begin{eqnarray*}
\sup_{\ep\in(0,\ep_0)}\sup_{t\in [0, T]}\mathbb{E}|\hat Y_{t}^{\ep}|^{k'\theta_4}\leq C_{T,k}(|x|^{2k'}+|y|^{k'\theta_4}+1).\label{Y2}
\end{eqnarray*}
\end{lemma}

\vskip 0.2cm
Now, we intend to estimate the difference process $Y_{t}^{\ep}-\hat{Y}_{t}^{\ep}$ and furthermore the difference process $X^{\ep}_t-\hat{X}_{t}^{\ep}$.
\begin{lemma} \label{DEY} Assume that either \textbf{$({\bf H}_1)$} with $\lambda_1=0$ and \textbf{$({\bf A}_{\tilde \theta_1})$} hold or \textbf{$({\bf H}_1)$} with $\lambda_1>0$ and \textbf{$({\bf A}_{k})$} with some $k>\tilde \theta_2$ hold, where $\tilde \theta_1=\max\{2\theta_6, 4\alpha_2\}$ and $\tilde \theta_2=\max\{2\theta_6, \theta_5\theta_4, 4\alpha_2, 2\alpha_1\theta_4\}$. Then for any $T>0$,
there exists a constant $C_{T,x,y}>0$ such that
$$
\sup_{\ep\in(0,1)}\sup_{t\in[0,T]}\mathbb{E}|Y_{t}^{\ep}-\hat{Y}_{t}^{\ep}|^{2}\leq C_{T,x,y}\delta^{\frac{1}{2}\wedge \gamma_2}.
$$
\end{lemma}
\begin{proof}
Note that
\begin{eqnarray*}
Y_{t}^{\ep}-\hat{Y}_{t}^{\ep}=\!\!\!\!\!\!\!\!&&\frac{1}{\ep}\int_{0}^{t} \left[f(s,X_{s}^{\ep}, Y_{s}^{\ep})-f(s(\delta), X_{s(\delta)}^{\ep},\hat{Y}_{s}^{\ep})\right]ds\\
&&+\frac{1}{\sqrt{\ep}}\int_{0}^{t}\left[g(s, X_{s}^{\ep}, Y_{s}^{\ep})-g(s(\delta), X_{s(\delta)}^{\ep},\hat{Y}_{s}^{\ep})\right]dW^{2}_s.
\end{eqnarray*}
For any $t\in[0,T]$,  by It\^{o}'s formula we have
\begin{eqnarray*}
\mathbb{E}|Y_{t}^{\ep}-\hat{Y}_{t}^{\ep}|^{2}
=\!\!\!\!\!\!\!\!&&\frac{2}{\ep}\int^t_0\mathbb{E}\Big[\langle f(s,X_{s}^{\ep},Y_{s}^{\ep})-f(s(\delta),X_{s(\delta)}^{\ep},\hat{Y}_{s}^{\ep}), Y_{s}^{\ep}-\hat{Y}_{s}^{\ep}\rangle\Big]ds \nonumber\\
&&+\frac{1}{\ep}\int^t_0\EE\|g(s,X_{s}^{\ep},Y_{s}^{\ep})-g(s(\delta),X_{s(\delta)}^{\ep},\hat{Y}_{s}^{\ep})\|^2 ds\nonumber\\
=\!\!\!\!\!\!\!\!&&\frac{1}{\ep}\int^t_0 \mathbb{E}\Big[2\langle f(s,X_{s}^{\ep},Y_{s}^{\ep})-f(s,X_{s}^{\ep},\hat{Y}_{s}^{\ep}), Y_{s}^{\ep}-\hat{Y}_{s}^{\ep}\rangle+\|g(s,X_{s}^{\ep},Y_{s}^{\ep})-g(s,X_{s}^{\ep},\hat{Y}_{s}^{\ep})\|^2\Big]ds\\
&&+\frac{2}{\ep}\int^t_0\EE\langle f(s,X_{s}^{\ep},\hat{Y}_{s}^{\ep})-f(s(\delta), X_{s(\delta)}^{\ep},\hat{Y}_{s}^{\ep}), Y_{s}^{\ep}-\hat{Y}_{s}^{\ep}\rangle ds\nonumber\\
&&+\frac{2}{\ep}\int^t_0\EE\langle g(s,X_{s}^{\ep},Y_{s}^{\ep})-g(s,X_{s}^{\ep},\hat{Y}_{s}^{\ep}), g(s,X_{s}^{\ep},\hat{Y}_{s}^{\ep})-g(s(\delta),X_{s(\delta)}^{\ep},\hat{Y}_{s}^{\ep})\rangle ds\\
&&+\frac{1}{\ep}\int^t_0\EE\|g(s,X_{s}^{\ep},\hat{Y}_{s}^{\ep})-g(s(\delta), X_{s(\delta)}^{\ep},\hat{Y}_{s}^{\ep})\|^2 ds.
\end{eqnarray*}
By condition \textbf{$({\bf H}_2)$} , we obtain
\begin{eqnarray}
\frac{d}{dt}\mathbb{E}|Y_{t}^{\ep}-\hat{Y}_{t}^{\ep}|^{2}
\leq\!\!\!\!\!\!\!\!&&\frac{-\beta}{\ep}\mathbb{E}|Y_{t}^{\ep}-\hat{Y}_{t}^{\ep}|^2\nonumber\\
&&+\frac{C_T}{\ep}\EE\left[(|X_{t}^{\ep}-X_{t(\delta)}^{\ep}|+\delta^{\gamma_2})\left(|\hat Y_{t}^{\ep}|^{\alpha_2}+|X_{t}^{\ep}|^{\alpha_1}+|X_{t(\delta)}^{\ep}|^{\alpha_1}+1\right)\cdot|Y_{t}^{\ep}-\hat{Y}_{t}^{\ep}|\right] \nonumber\\
&&+\frac{C_T}{\ep}\EE|Y_{t}^{\ep}-\hat{Y}_{t}^{\ep}|(|X_{t}^{\ep}-X_{t(\delta)}^{\ep}|+\delta^{\gamma_2})+\frac{C_T}{\ep}\EE|X_{t}^{\ep}-X_{t(\delta)}^{\ep}|^2+\frac{C_T\delta^{2\gamma_2}}{\ep}\nonumber\\
\leq\!\!\!\!\!\!\!\!&&
-\frac{\beta}{2\ep}\mathbb{E}|Y_{t}^{\ep}-\hat{Y}_{t}^{\ep}|^{2}+\frac{C_T}{\ep}\mathbb{E}|X_{t}^{\ep}-X_{t(\delta)}^{\ep}|^2+\frac{C_T\delta^{2\gamma_2}}{\ep}\nonumber\\
&&+\frac{C_T}{\ep}\mathbb{E}\left[(|X_{t}^{\ep}-X_{t(\delta)}^{\ep}|+\delta^{\gamma_2})\left(|\hat Y_{t}^{\ep}|^{\alpha_2}+|X_{t}^{\ep}|^{\alpha_1}+|X_{t(\delta)}^{\ep}|^{\alpha_1}+1\right)\cdot|Y_{t}^{\ep}-\hat{Y}_{t}^{\ep}|\right].\nonumber
\end{eqnarray}
The comparison theorem implies that
\begin{eqnarray*}
&&\mathbb{E}|Y_{t}^{\ep}-\hat{Y}_{t}^{\ep}|^{2}\\
\leq\!\!\!\!\!\!\!\!&& \frac{C_T}{\ep}\int^t_0 e^{-\frac{\beta(t-s)}{2\ep}}\mathbb{E}\left[(|X_{s}^{\ep}-X_{s(\delta)}^{\ep}|+\delta^{\gamma_2})\left(|\hat Y_{s}^{\ep}|^{\alpha_2}+|X_{s}^{\ep}|^{\alpha_1}+|X_{s(\delta)}^{\ep}|^{\alpha_1}+1\right)\cdot|Y_{s}^{\ep}-\hat{Y}_{s}^{\ep}|\right]ds\\
&&+\frac{C}{\ep}\int^t_0 e^{-\frac{\beta(t-s)}{2\ep}}\mathbb{E}|X_{s}^{\ep}-X_{s(\delta)}^{\ep}|^2ds+\frac{C_T}{\ep}\int^t_0 e^{-\frac{\beta(t-s)}{2\ep}}\delta^{2\gamma_2}ds\\
\leq\!\!\!\!\!\!\!\!&&\frac{C_T}{\ep}\int^t_0 e^{-\frac{\beta(t-s)}{2\ep}}\left(\mathbb{E}|X_{s}^{\ep}-X_{s(\delta)}^{\ep}|^{2}+\delta^{2\gamma_2}\right)^{1/2}\\
&&\quad\quad\quad\left[\EE\left(|\hat Y_{s}^{\ep}|^{4\alpha_2}+|X_{s}^{\ep}|^{4\alpha_1}+|X_{s(\delta)}^{\ep}|^{4\alpha_1}+1\right)\EE\left(|Y_{s}^{\ep}-\hat{Y}_{s}^{\ep}|^4\right)\right]^{1/4}ds\\
&&+\frac{C}{\ep}\int^t_0 e^{-\frac{\beta(t-s)}{2\ep}}\mathbb{E}|X_{s}^{\ep}-X_{s(\delta)}^{\ep}|^2ds+C_T\delta^{2\gamma_2}.\\
\end{eqnarray*}
Hence, by Lemma \ref{PMY} and \ref{COX}, we have
\begin{eqnarray*}
\mathbb{E}|Y_{t}^{\ep}-\hat{Y}_{t}^{\ep}|^{2}\leq\!\!\!\!\!\!\!\!&&C_{T,x,y}\delta^{\frac{1}{2}\wedge \gamma_2}.
\end{eqnarray*}
The proof is complete.
\end{proof}

In order to estimate the difference process $X_{t}^{\ep}-\hat{X}_{t}^{\ep}$. We first construct the following stopping time, for fixed $\ep\in (0,\ep_0), R\geq R_0, M\geq 0$,
\begin{eqnarray*}
\tau^{\ep}_{R,M}:=\!\!\!\!\!\!\!\!&&\inf\left\{t\geq 0: |X_{t}^{\ep}|+\int^t_0|Y^{\ep}_s|^{2\theta_2}ds+\int^t_0|\hat Y^{\ep}_s|^{4\theta_1\vee 2\theta_2}ds+\int^t_0 \left[K_s(1)\right]^2ds\geq R\right\}\\
&&\wedge \inf\left\{t\geq 0: \int^t_0 |K_s(R)|^4ds\geq M\right\},
\end{eqnarray*}
and $\inf\{\emptyset\}:=\infty$.
\begin{lemma} \label{DEX} Assume that either \textbf{$({\bf H}_1)$} with $\lambda_1=0$ and \textbf{$({\bf A}_{\tilde \theta_1})$} hold or \textbf{$({\bf H}_1)$} with $\lambda_1>0$ and \textbf{$({\bf A}_{k})$} with some $k>\tilde \theta_2$ hold, where $\tilde \theta_1=\max\{2\theta_6, 4\alpha_2\}$ and $\tilde \theta_2=\max\{2\theta_6, \theta_5\theta_4, 4\alpha_2, 2\alpha_1\theta_4\}$. Then for any $T, M>0$ and $R\geq R_0$, there exists a constant $C_{T,R,M}>0$ such that
\begin{eqnarray*}
\mathbb{E}\Big(\sup_{t\in [0, T\wedge\tau^{\ep}_{R,M}]}|X_{t}^{\ep}-\hat{X}_{t}^{\ep}|^2\Big)\leq C_{T,R,M}\delta^{\gamma},
\end{eqnarray*}
where $\gamma=\min\{2\gamma_1,\gamma_2, 1/2\}$.
\end{lemma}

\begin{proof}
Recall that
\begin{eqnarray*}
X^{\ep}_t=x+\int^t_0 b(s, X^{\ep}_s, Y^{\ep}_s)ds+\int^t_0\sigma(s, X^{\ep}_s)dW^{1}_s
\end{eqnarray*}
and
\begin{eqnarray*}
\hat{X}^{\ep}_t=x+\int^t_0 b(s(\delta), X^{\ep}_{s(\delta)}, \hat{Y}^{\ep}_s)ds+\int^t_0\sigma(s, X^{\ep}_s)dW^{1}_s.
\end{eqnarray*}
Then we have
\begin{eqnarray*}
X^{\ep}_t-\hat{X}^{\ep}_t=\int^t_0\big[b(s, X^{\ep}_s, Y^{\ep}_s)-b(s(\delta), X^{\ep}_{s(\delta)}, \hat{Y}^{\ep}_s)\big]ds.
\end{eqnarray*}
By Lemma \ref{COX} and \ref{DEY} we have
\begin{eqnarray*}
&&\EE\left(\sup_{t\in[0, T\wedge\tau^{\ep}_{R,M}]}|X^{\ep}_{t} -\hat{X}^{\ep}_{t}|^2\right)\\
\leq\!\!\!\!\!\!\!\!&&\EE\left[\int^{T\wedge\tau^{\ep}_{R,M}}_0\left|b(s, X^{\ep}_s, Y^{\ep}_s)-b(s(\delta), X^{\ep}_{s(\delta)}, \hat{Y}^{\ep}_s)\right| ds\right]^2\nonumber\\
\leq\!\!\!\!\!\!\!\!&&C\EE\left[\int^{T\wedge\tau^{\ep}_{R,M}}_0\!\!\left|b(s, X^{\ep}_s, Y^{\ep}_s)-b(s, X^{\ep}_s, \hat Y^{\ep}_s)\right|ds\right]^2\!\!\!+\!\!C\EE\left[\int^{T\wedge\tau^{\ep}_{R,M}}_0\!\!\left|b(s, X^{\ep}_s, \hat Y^{\ep}_s)-b(s, X^{\ep}_{s(\delta)}, \hat Y^{\ep}_s)\right|ds\right]^2\\
&&+C\EE\left[\int^{T\wedge\tau^{\ep}_{R,M}}_0\left|b(s, X^{\ep}_{s(\delta)}, \hat Y^{\ep}_s)-b(s(\delta), X^{\ep}_{s(\delta)}, \hat{Y}^{\ep}_s)\right|ds\right]^2\nonumber\\
\leq\!\!\!\!\!\!\!\!&&
C\EE\left[\int^{T\wedge\tau^{\ep}_{R,M}}_0 |Y^{\ep}_s-\hat Y^{\ep}_s|^2ds \cdot\int^{T\wedge\tau^{\ep}_{R,M}}_0 \left(|Y^{\ep}_s|^{2\theta_2}+|\hat Y^{\ep}_s|^{2\theta_2}+|X^{\ep}_s|^{2\theta_3}+[K_{s}(1)]^2\right)ds\right]\\
&&+C\EE\left[\int^{T\wedge\tau^{\ep}_{R,M}}_0|X^{\ep}_s-X^{\ep}_{s(\delta)}|^2ds\left(\int^{T\wedge\tau^{\ep}_{R,M}}_0 [K_s(R)]^4ds\right)^{1/2}\left(\int^{T\wedge\tau^{\ep}_{R,M}}_0\left(1+|\hat Y^{\ep}_s|^{4\theta_1}\right)ds\right)^{1/2}\right]\nonumber\\
&&+\delta^{2\gamma_1}C_T\EE\int^{T\wedge\tau^{\ep}_{R,M}}_0 \left(|X^{\ep}_{s(\delta)}|^{2\theta_3}+|\hat Y^{\ep}_s|^{2\theta_2}\right)ds+\delta^{2\gamma_1}C_T\EE Z_{T}^2 \\
\leq\!\!\!\!\!\!\!\!&& C_{T,R}\EE\int^{T\wedge\tau^{\ep}_{R,M}}_0 |Y^{\ep}_s-\hat Y^{\ep}_s|^2ds+C_{R,M}\EE\int^{T\wedge\tau^{\ep}_{R,M}}_0|X^{\ep}_s-X^{\ep}_{s(\delta)}|^2ds+C_{T,R,M}\delta^{2\gamma_1}\\
\leq\!\!\!\!\!\!\!\!&& C_{T,R,M}\delta^{\min\{2\gamma_1,\gamma_2, 1/2\}} . \label{bs}
\end{eqnarray*}
The proof is complete.
\end{proof}

\subsection{The frozen equation}
We first introduce the frozen equation associated to the fast motion for fixed $t>0$ and fixed slow component $x\in \RR^n$.
 \begin{equation}\left\{\begin{array}{l}\label{FEQ}
\displaystyle
dY_{s}=f(t, x,Y_{s})ds+g(t, x,Y_{s})d\tilde{W}_{s}^{2},\\
Y_{0}=y,\\
\end{array}\right.
\end{equation}
where $\{\tilde {W}_{s}^{2}\}_{s\geq 0}$ is a $d_2$-dimensional Brownian motion on another complete probability space $(\tilde{\Omega}, \tilde{\mathscr{F}}, \tilde{\mathbb{P}})$ and $\{\tilde{\mathscr{F}}_{s},s\geq 0\}$ is the natural filtration generated by $\{\tilde{W}_{s}^{2}\}_{s\geq 0}$.
If \textbf{$({\bf H}_2)$} and \textbf{$({\bf A}_2)$} hold, then it is easy to prove for any fixed $t>0$, $x\in \RR^n$ and any initial data $y\in \RR^m$,
Eq. $(\ref{FEQ})$ has a unique strong solution $\{Y_{s}^{t,x,y}\}_{s\geq 0}$, which is a time homogeneous  Markov process. Let $\{P^{t,x}_s\}_{s\geq 0}$ be the transition semigroup of $\{Y_{s}^{t,x,y}\}_{s\geq 0}$, $i.e.$ for any bounded measurable function $\varphi:\RR^m\rightarrow \mathbb{R}$,
$$
P^{t,x}_s\varphi(y):=\tilde \EE\varphi(Y_{s}^{t,x,y}), \quad y\in\RR^m, s\geq 0,
$$
where $\tilde \EE$ is the expectation on $(\tilde{\Omega}, \tilde{\mathscr{F}}, \tilde{\mathbb{P}})$.
\begin{lemma}\label{L 3.6}
Suppose that \textbf{$({\bf A}_k)$} holds for some $k\geq 2$. Then there exists $\tilde \beta_k>0$ such that for any $x\in\RR^n, y\in \RR^m$, $s\geq 0$ and $T>0$ with $t\in [0, T]$,
\begin{eqnarray}
\tilde{\mathbb{E}}|Y_{s}^{t,x,y}|^k\leq e^{-\tilde \beta_k s}|y|^k+C_{T,k}(1+|x|^{\frac{2k}{\theta_4}}). \label{F3.4}
\end{eqnarray}
\end{lemma}
\begin{proof}
Note that
$$
Y_{s}^{t,x,y}=y+\int^s_0 f(t,x, Y_{r}^{t,x,y})dr+\int^s_0 g(t,x,Y^{t,x,y}_{r})d\tilde {W}_{r}^{2}.
$$
By the It\^o's formula we have
\begin{eqnarray*}
\tilde{\mathbb{E}}|Y_{s}^{t,x,y}|^{k}=\!\!\!\!\!\!\!\!&&k\int^s_0\tilde{\mathbb{E}}\left[| Y_{r}^{t,x,y}|^{k-2}\langle f(t,x,Y_{r}^{t,x,y}),Y_{r}^{t,x,y}\rangle\right]dr+\frac{k}{2}\int^s_0\tilde{\mathbb{E}}\left[|Y_{r}^{t,x,y}|^{k-2}\|g(t,x,Y_{r}^{t,x,y})\|^2\right]dr \nonumber \\
 \!\!\!\!\!\!\!\!&& +\frac{k(k-2)}{2}\int^s_0\tilde{\mathbb{E}}\left[|Y_{r}^{t,x,y}|^{k-4}\cdot|\langle Y_{r}^{t,x,y}, g(t,x,Y_{r}^{t,x,y})\rangle|^2\right]dr.
\end{eqnarray*}
Then assumption \textbf{$({\bf A}_{k})$} yields that there exists $\tilde \beta_k>0$ such that for any $t\in [0, T]$
\begin{eqnarray*}
\frac{d}{ds}\tilde{\mathbb{E}}|Y_{s}^{t,x,y}|^{k}\leq\!\!\!\!\!\!\!\!&&\frac{k}{2}\tilde{\mathbb{E}}\left[| Y_{s}^{t,x,y}|^{k-2}\left(2\langle f(t,x,Y_{s}^{t,x,y}),Y_{s}^{t,x,y}\rangle+(k-1)\|g(t,x, Y_{s}^{t,x,y})\|^2\right)\right]\\
 \leq\!\!\!\!\!\!\!\!&&-\tilde \beta_k\tilde{\mathbb{E}}|Y_{s}^{t,x,y}|^{k}
+C_{T,k}\left(|x|^{\frac{2k}{\theta_4}}+1\right).
\end{eqnarray*}
Hence, by the comparison theorem we have
\begin{eqnarray*}
\tilde{\mathbb{E}}|Y_{s}^{t,x,y}|^{k}\leq\!\!\!\!\!\!\!\!&&|y|^{k}e^{-\tilde \beta_k s}+C_{T,k}(1+|x|^{\frac{2k}{\theta_4}}) \int^s_0 e^{-\tilde \beta_k(s-r)}dr\nonumber\\
\leq\!\!\!\!\!\!\!\!&&|y|^{k}e^{-\tilde \beta_k s}+C_{T,k}(1+|x|^{\frac{2k}{\theta_4}}).
\end{eqnarray*}
The proof is complete.
\end{proof}

\begin{lemma}\label{L 3.61}
There exists $\beta>0$ such that for any $t,s\geq 0$, $x\in\RR^n, y_1,y_2\in \RR^m$,
\begin{eqnarray*}
\tilde{\mathbb{E}}|Y_{s}^{t,x,y_1}-Y_{s}^{t,x,y_2}|^2\leq e^{-\beta s}|y_1-y_2|^2.
\end{eqnarray*}
\end{lemma}
\begin{proof}
Note that
\begin{eqnarray*}
Y_{s}^{t,x,y_1}-Y_{s}^{t,x,y_2}=\!\!\!\!\!\!\!\!&&y_1-y_2+\int^s_0  f(t,x, Y_{r}^{t,x,y_1})-f(t,x, Y_{r}^{t,x,y_2}) dr\\
&&+\int^s_0 \left[g(t,x,Y^{t,x,y_1}_{r})-g(t,x,Y^{t,x,y_2}_{r})\right]d\tilde{W}_{r}^{2}.
\end{eqnarray*}
By It\^o's formula we obtain
\begin{eqnarray*}
\tilde{\mathbb{E}}|Y_{s}^{t,x,y_1}-Y_{s}^{t,x,y_2}|^{2}=\!\!\!\!\!\!\!\!&&\int^s_0\tilde{\mathbb{E}}\left[2\langle f(t,x, Y_{r}^{t,x,y_1})-f(t,x, Y_{r}^{t,x,y_2}), Y_{r}^{t,x,y_1}-Y_{r}^{t,x,y_2}\rangle \right]dr\\
&&+\int^s_0\tilde{\mathbb{E}}\|g(t,x,Y^{t,x,y_1}_{r})-g(t,x,Y^{t,x,y_2}_{r})\|^2 dr. \nonumber
\end{eqnarray*}
Then condition \eref{sm} in \textbf{$({\bf H}_{2})$} yields that there exist $\beta>0$ and $C\geq 0$ such that
\begin{eqnarray*}
\frac{d}{ds}\tilde{\mathbb{E}}|Y_{s}^{t,x,y_1}-Y_{s}^{t,x,y_2}|^{2}
 \leq\!\!\!\!\!\!\!\!&&-\beta\tilde{\mathbb{E}}|Y_{s}^{t,x,y_1}-Y_{s}^{t,x,y_2}|^{2}.
\end{eqnarray*}
The comparison theorem implies that
\begin{eqnarray*}
\tilde{\mathbb{E}}|Y_{s}^{t,x,y_1}-Y_{s}^{t,x,y_2}|^{2}\leq\!\!\!\!\!\!\!\!&&e^{-\beta s}|y_1-y_2|^2.
\end{eqnarray*}
The proof is complete.
\end{proof}

\begin{proposition}\label{invariant}  Suppose that \textbf{$({\bf A}_k)$} holds for some $k\geq 2$. For any $t\in [0, T]$, $x\in\RR^n$, $\{P^{t,x}_s\}_{s\geq 0}$ has a unique invariant measure $\mu^{t,x}$. Moreover,
\begin{eqnarray}
\int_{\RR^m}|z|^k\mu^{t,x}(dz)\leq C_{T,k}(1+|x|^{\frac{2k}{\theta_4}}). \label{F3.5}
\end{eqnarray}
\end{proposition}
\begin{proof} We first check \eref{F3.5}. If $\mu^{t,x}$ is an invariant measure of $\{P^{t,x}_s\}_{s\geq 0}$,  it follows from Lemma \ref{L 3.6} that for all $s>0$
\begin{eqnarray*}
\int_{\RR^m}|z|^k\mu^{t,x}(dz)=\!\!\!\!\!\!\!\!&&\int_{\RR^m}\tilde{\mathbb{E}}|Y_{s}^{t,x,z}|^k\mu^{t,x}(dz)\\
\leq\!\!\!\!\!\!\!\!&&\int_{\RR^m}\left[e^{-\tilde \beta_k s}|z|^k+C_{T,k}(1+|x|^{\frac{2k}{\theta_4}})\right]\mu^{t,x}(dz)\\
=\!\!\!\!\!\!\!\!&&e^{-\tilde \beta_k s}\int_{\RR^m}|z|^k\mu^{t,x}(dz)+C_{T,k}(1+|x|^{\frac{2k}{\theta_4}}).
\end{eqnarray*}
Taking $s$ large enough such that $e^{-\tilde \beta_k s}\leq \frac{1}{2}$, we obtain \eref{F3.5}.

The estimate \eref{F3.4} and the classical Bogoliubov-Krylov argument imply the existence of invariant measures. For the uniqueness, it is sufficient to prove that for any Lipschitz function $\varphi(x):\RR^m\rightarrow \RR$ and any invariant measure $\mu^{t,x}$ we  have
$$
\left|P^{t,x}_s\varphi(y)-\int_{\RR^m}\varphi(z)\mu^{t,x}(dz)\right|\leq C_T Lip(\varphi)e^{-\frac{\beta s}{2} }(1+|x|^{\frac{2}{\theta_4}}+|y|),\quad s\geq 0,
$$
where $Lip(\varphi)=\sup_{x\neq y}\frac{|\varphi(x)-\varphi(y)|}{|x-y|}$.

In fact, by Lemma \ref{L 3.61} and \eref{F3.5}, we have
\begin{eqnarray*}
\left|P^{t,x}_s\varphi(y)-\int_{\RR^m}\varphi(z)\mu^{t,x}(dz)\right|\leq\!\!\!\!\!\!\!\!&&\int_{\RR^m}\left|\tilde{\mathbb{E}}\varphi(Y^{t,x,y}_s)-\tilde \EE\varphi(Y^{t,x,z}_s)\right|\mu^{t,x}(dz)\\
\leq\!\!\!\!\!\!\!\!&&Lip(\varphi)\int_{\RR^m}\tilde{\mathbb{E}}\left|Y^{t,x,y}_s-Y^{t,x,z}_s\right|\mu^{t,x}(dz)\\
\leq\!\!\!\!\!\!\!\!&&Lip(\varphi)\int_{\RR^m}e^{-\frac{\beta s}{2} }|y-z|\mu^{t,x}(dz)\\
\leq\!\!\!\!\!\!\!\!&&C_T Lip(\varphi)e^{-\frac{\beta s}{2} }(1+|x|^{\frac{2}{\theta_4}}+|y|).
\end{eqnarray*}
Hence the proof is complete.
\end{proof}

\begin{proposition}\label{Rem 4.1}  Suppose that \textbf{$({\bf A}_{2\theta_2})$} holds. Then for any $T>0$, there exists $C_T>0$ such that any $x\in\RR^n, y\in \RR^m$, $t\in [0, T]$ and $s\geq 0$ ,
\begin{eqnarray}
\left|\tilde \EE b(t, x, Y^{t,x,y}_s)-\int_{\RR^m}b(t,x,z)\mu^{t,x}(dz)\right|\leq\!\!\!\!\!\!\!\!&& C_T e^{-\frac{\beta s}{2}}\left[(K_t(1))^2+1+|x|^{\theta}+|y|^{\theta_2+1}\right], \label{ergodicity}
\end{eqnarray}
where $\theta=\max\{\frac{2\theta_2+2}{\theta_4}, \frac{\theta_3\theta_4+2}{\theta_4}, \frac{\theta_3(\theta_2+1)}{\theta_2}\}$.
\end{proposition}
\begin{proof}
By Lemma \ref{L 3.6} and \ref{L 3.61} and Proposition \ref{invariant},   for any $s\geq 0$ we have
\begin{eqnarray*}
&&\left|\tilde \EE b(t, x, Y^{t,x,y}_s)-\int_{\RR^m}b(t, x,z)\mu^{t,x}(dz)\right|\\
=\!\!\!\!\!\!\!\!&& \left|\int_{\RR^m}\tilde \EE b(t, x, Y^{t,x,y}_s)-\tilde \EE b(t, x, Y^{t,x,z}_s)\mu^{t,x}(dz)\right|\\
\leq\!\!\!\!\!\!\!\!&& C_T\int_{\RR^m}\tilde \EE\left[\left| Y^{t,x,y}_s-Y^{t,x,z}_s\right|(|Y^{t,x,y}_s|^{\theta_2}+|Y^{t,x,z}_s|^{\theta_2}+|x|^{\theta_3}+K_t(1))\right]\mu^{t,x}(dz)\\
\leq\!\!\!\!\!\!\!\!&& C\int_{\RR^m}\left[\tilde \EE\left(\left| Y^{t,x,y}_s-Y^{t,x,z}_s\right|^2\right)\right]^{1/2}\left[\tilde \EE\left(|Y^{t,x,y}_s|^{2\theta_2}+|Y^{t,x,z}_s|^{2\theta_2}+|x|^{2\theta_3}+[K_t(1)]^2\right)\right]^{1/2}\mu^{t,x}(dz)\\
\leq\!\!\!\!\!\!\!\!&&Ce^{-\frac{\beta s}{2}}\int_{\RR^m}|z-y|\left[|z|^{\theta_2}+|y|^{\theta_2}+|x|^{\frac{2\theta_2}{\theta_4}}+|x|^{\theta_3}+K_t(1)+1\right]\mu^{t,x}(dz)\\
\leq\!\!\!\!\!\!\!\!&&C_T e^{-\frac{\beta s}{2}}\left[K_t(1)(|x|^{2/ \theta_4}+|y|+1)+|x|^{\theta}+|y|^{\theta_2+1}\right]\\
\leq\!\!\!\!\!\!\!\!&&C_T e^{-\frac{\beta s}{2}}\left[(K_t(1))^2+1+|x|^{\theta}+|y|^{\theta_2+1}\right],
\end{eqnarray*}
where $\theta=\max\{\frac{2\theta_2+2}{\theta_4}, \frac{\theta_3\theta_4+2}{\theta_4}, \frac{\theta_3(\theta_2+1)}{\theta_2}\}$. The proof is complete.
\end{proof}

\begin{lemma} \label{L3.7} Suppose that \textbf{$({\bf A}_{2\alpha_2})$} holds.  Then for any $T>0$,  there exists a constant $C_T>0$ such that for all $x_1,x_2\in\RR^n$, $y\in\RR^m$, $t_1,t_2\in [0, T]$ and $s\geq 0$,
\begin{eqnarray*}
\tilde{\mathbb{E}}|Y^{t_1,x_1,y}_s-Y^{t_2,x_2,y}_s|^2\leq C_T(1+|x_1|^{2\alpha_1}+|x_2|^{\max{\{\frac{4\alpha_2}{\theta_4}, 2\alpha_1\}}}+|y|^{2\alpha_2})\left(|x_1-x_2|^2+|t_1-t_2|^{2\gamma_2}\right).
\end{eqnarray*}
\end{lemma}
\begin{proof}
Note that
\begin{eqnarray*}
Y^{t_1,x_1,y}_s-Y^{t_2,x_2,y}_s=\!\!\!\!\!\!\!\!&&\int^s_0 f(t_1,x_1, Y^{t_1,x_1,y}_r)-f(t_2,x_2, Y^{t_2,x_2,y}_r)dr\\
&&+\int^s_0 g(t_1,x_1, Y^{t_1,x_1,y}_r)-g(t_2,x_2, Y^{t_2,x_2,y}_r)d \tilde {W}^2_r.
\end{eqnarray*}
By It\^{o}'s formula we have
\begin{eqnarray*}
&&\tilde{\mathbb{E}}|Y^{t_1,x_1,y}_s-Y^{t_2,x_2,y}_s|^2\\
=\!\!\!\!\!\!\!\!&&\int^s_0\tilde{\mathbb{E}}\left[2\langle f(t_1,x_1, Y^{t_1,x_1,y}_r)-f(t_2,x_2, Y^{t_2,x_2,y}_r), Y^{t_1,x_1,y}_r-Y^{t_2,x_2,y}_r\rangle\right.\\
&&\left.+\|g(t_1,x_1, Y^{t_1,x_1,y}_r)-g(t_2,x_2, Y^{t_2,x_2,y}_r)\|^2\right]dr\\
=\!\!\!\!\!\!\!\!&&\int^s_0\tilde{\mathbb{E}}\left[2\left\langle f(t_1,x_1, Y^{t_1,x_1,y}_r)-f(t_1, x_1, Y^{t_2,x_2,y}_r), Y^{t_1,x_1,y}_r-Y^{t_2,x_2,y}_r\right\rangle\right.\\
&&\left.+\left\|g(t_1,x_1, Y^{t_1,x_1,y}_r)-g(t_1, x_1, Y^{t_2,x_2,y}_r)\right\|^2\right]dr\\
&&+\int^s_0\tilde{\mathbb{E}}\left[2\left\langle f(t_1,x_1, Y^{t_2,x_2,y}_r)-f(t_2,x_2, Y^{t_2,x_2,y}_r), Y^{t_1, x_1,y}_r-Y^{t_2, x_2,y}_r\right\rangle\right]dr\\
&&+\int^s_0\tilde{\mathbb{E}}\left\|g(t_1, x_1, Y^{t_2, x_2,y}_r)-g(t_2, x_2, Y^{t_2, x_2,y}_r)\right\|^2 dr\\
&&+\int^s_0\tilde{\mathbb{E}}\left[2\left\langle g(t_1, x_1, Y^{t_1,x_1,y}_r)-g(t_1, x_1, Y^{t_2, x_2,y}_r), g(t_1,x_1, Y^{t_2, x_2,y}_r)-g(t_2,x_2, Y^{t_2,x_2,y}_r)\right\rangle\right]dr.
\end{eqnarray*}
Then by Young's inequality and \eref{sm}, there exists $\beta>0$ such that
\begin{eqnarray*}
&&\frac{d}{ds}\tilde{\mathbb{E}}|Y^{t_1,x_1,y}_s-Y^{t_2,x_2,y}_s|^2\\
\leq\!\!\!\!\!\!\!\!&& -\beta\tilde{\mathbb{E}}\left|Y^{t_1,x_1,y}_s-Y^{t_2,x_2,y}_s\right|^2+C_T(|x_1-x_2|^2+|t_1-t_2|^{2\gamma_2})\\
&&+C_T\tilde{\mathbb{E}}\left[\left|Y^{t_1,x_1,y}_s-Y^{t_2,x_2,y}_s\right|\left(|x_1-x_2|+|t_1-t_2|^{\gamma_2}\right)\right]\\
&&+C_T\tilde{\mathbb{E}}\left[(1+|Y^{t_2,x_2,y}_s|^{\alpha_2}+|x_1|^{\alpha_1}+|x_2|^{\alpha_1})\left|Y^{t_1,x_1,y}_s-Y^{t_2,x_2,y}_s\right|\right](|x_1-x_2|+|t_1-t_2|^{\gamma_2})\\
\leq\!\!\!\!\!\!\!\!&& -\frac{\beta}{2}\tilde{\mathbb{E}}\left|Y^{t_1,x_1,y}_s-Y^{t_2,x_2,y}_s\right|^2+C_T\tilde{\mathbb{E}}\left(1+|Y^{t_2,x_2,y}_s|^{2\alpha_2}+|x_1|^{2\alpha_1}+|x_2|^{2\alpha_1}\right)\left(|x_1-x_2|^2+|t_1-t_2|^{2\gamma_2}\right)\\
\leq\!\!\!\!\!\!\!\!&& -\frac{\beta}{2}\tilde{\mathbb{E}}\left|Y^{t_1,x_1,y}_s-Y^{t_2,x_2,y}_s\right|^2\!\!+\!C_T(1+|x_1|^{2\alpha_1}+|x_2|^{\max{\{\frac{4\alpha_2}{\theta_4}, 2\alpha_1\}}}+|y|^{2\alpha_2})\left(|x_1-x_2|^2+|t_1-t_2|^{2\gamma_2}\right).
\end{eqnarray*}
Hence, the comparison theorem yields that
\begin{eqnarray*}
\tilde{\mathbb{E}}|Y^{x_1,y}_t-Y^{x_2,y}_t|^2\leq C_T(1+|x_1|^{2\alpha_1}+|x_2|^{\max{\{\frac{4\alpha_2}{\theta_4}, 2\alpha_1\}}}+|y|^{2\alpha_2})\left(|x_1-x_2|^2+|t_1-t_2|^{2\gamma_2}\right).
\end{eqnarray*}
The proof is complete.
\end{proof}

\vskip 0.2cm
\subsection{The averaged equation}
Now we introduce the following averaged equation
\begin{equation}\left\{\begin{array}{l}
\displaystyle d\bar{X}_{t}=\bar{b}(t, \bar{X}_{t})dt+\sigma(t, \bar{X}_{t})dW^{1}_t,\\
\bar{X}_{0}=x\in \RR^n. \end{array}\right. \label{3.1}
\end{equation}
Here
\begin{eqnarray*}
\bar{b}(t,x)=\int_{\RR^m}b(t,x,y)\mu^{t,x}(dy),
\end{eqnarray*}
where $\mu^{t,x}$ is the unique invariant measure for Eq.(\ref{FEQ}).

\vspace{0.3cm}
The following lemma gives the  existence, uniqueness and uniformly estimates of solutions for Eq. \eref{3.1}. The proof will be presented in the Appendix.

\begin{lemma} \label{PMA} Suppose that \textbf{$({\bf A}_{\tilde \theta})$} holds with $\tilde \theta=\max\{2\theta_2, \theta_1, \theta_4, 2\alpha_2\}$. Then Eq.(\ref{3.1}) has a unique solution. Furthermore, for any $x\in\RR^n$, $p\geq2$ and $T>0$,
\begin{eqnarray}
\mathbb{E}\left(\sup_{t\in [0, T]}|\bar{X}_{t}|^{p}\right)\leq C_{T,p}(1+|x|^{p}),\label{3.9}
\end{eqnarray}
where  $C_{T,p}$ is some positive constant.
\end{lemma}

\vskip 0.2cm
\subsection{The Proof of the main result} In this part, we intend to give a complete proof for our main result, $i.e.$
the slow component process $X_{t}^{\ep}$ converges strongly to the solution $\bar{X}_{t}$ of the averaged equation.
We first estimate the error between the auxiliary process $\hat{X}_{t}^{\ep}$ and the solution $\bar{X}_{t}$ of  the averaged equation  before a stopping time.

\begin{lemma} \label{ESX} Assume either \textbf{$({\bf H}_1)$} with $\lambda_1=0$ and \textbf{$({\bf A}_{\tilde \theta_1})$} hold or \textbf{$({\bf H}_1)$} with $\lambda_1>0$ and \textbf{$({\bf A}_{k})$} with some $k>\tilde \theta_2$ hold, where $\tilde \theta_1=\max\{\theta_1, 2\theta_2+2, 2\theta_6, 4\alpha_2\}$ and $\tilde \theta_2=\max\{\theta_1, 2\theta_2+2, 2\theta_6, \theta_5\theta_4, 4\alpha_2, 2\alpha_1\theta_4\}$. Then for any $T>0$, $R\geq R_0$ and $M>0$, there exists a constant $C_{T,R,M,x,y}>0$ such that
\begin{eqnarray*}
\mathbb{E}\left(\sup_{t\in [0, T\wedge\tilde{\tau}^{\ep}_{R,M}]}|\hat{X}_{t}^{\ep}-\bar{X}_{t}|^2\right)\leq C_{T,R,M,x,y}\left(\frac{\ep}{\delta}+\delta^{\gamma}\right),
\end{eqnarray*}
where $\tilde{\tau}^{\ep}_{R,M}:=\inf\{t\geq 0: |\bar{X}_t|\geq R\}\wedge \tau^{\ep}_{R,M}$ and $\gamma=\min\{2\gamma_1, \gamma_2, 1/2\}$.
\end{lemma}

\begin{proof}
Recall that
\begin{eqnarray*}
\hat{X}_{t}^{\ep}-\bar{X}_{t}&=&\int_{0}^{t}\left[b(s(\delta), X_{s(\delta)}^{\ep},\hat{Y}_{s}^{\ep})-\bar{b}(s,\bar{X}_{s})\right]ds
+\int_{0}^{t}\left[\sigma(s, X^{\ep}_{s})-\sigma(s, \bar{X}_{s})\right]dW^{1}_s\\
&=&\int_{0}^{t}\left[b(s(\delta), X_{s(\delta)}^{\ep},\hat{Y}_{s}^{\ep})-\bar{b}(s(\delta), X^{\ep}_{s(\delta)})\right]ds+\int_{0}^{t}\left[\bar{b}(s(\delta), X^{\ep}_{s(\delta)})-\bar{b}(s, X^{\ep}_{s(\delta)})\right]ds\\
&&+\int_{0}^{t}\left[\bar{b}(s, X^{\ep}_{s(\delta)})-\bar{b}(s, X^{\ep}_{s})\right]ds+\int_{0}^{t}\left[\bar{b}(s, X_{s}^{\ep})-\bar{b}(s, \bar{X}_s)\right]ds\\
&&+\int_{0}^{t}\left[\sigma(s, X^{\ep}_{s})-\sigma(s, \bar{X}_{s})\right]dW^{1}_s.
\end{eqnarray*}
Then it is easy to see that
\begin{eqnarray}
&&\EE\left(\sup_{t\in[0,T\wedge \tilde{\tau}^{\ep}_{R,M}]}|\hat{X}_{t}^{\ep}-\bar{X}_{t}|^2\right)\nonumber\\
\leq\!\!\!\!\!\!\!\!&&C\EE\left[\sup_{t\in[0,T\wedge \tilde{\tau}^{\ep}_{R,M}]}\left|\int_{0}^{t}b(s(\delta), X_{s(\delta)}^{\ep},\hat{Y}_{s}^{\ep})-\bar{b}(s(\delta),X^{\ep}_{s(\delta)})ds\right|^2\right]\nonumber\\
&&+\EE\left[\int_{0}^{T\wedge \tilde{\tau}^{\ep}_{R,M}}\left|\bar{b}(s(\delta), X^{\ep}_{s(\delta)})-\bar{b}(s, X^{\ep}_{s(\delta)})\right|ds\right]^2\nonumber\\
&&+\EE\left[\int_{0}^{T\wedge \tilde{\tau}^{\ep}_{R,M}}\left|\bar{b}(s, X^{\ep}_{s(\delta)})-\bar{b}(s, X^{\ep}_{s})\right|ds\right]^2
+\EE\left[\int_{0}^{T\wedge \tilde{\tau}^{\ep}_{R,M}}\left|\bar{b}(s,X^{\ep}_{s})-\bar{b}(s,\bar{X}_{s})\right|ds\right]^2\nonumber\\
&&+C\EE\int_{0}^{T\wedge \tilde{\tau}^{\ep}_{R,M}}\|\sigma(s,X_{s}^{\ep})-\sigma(s,\bar{X}_{s})\|^2ds   \nonumber\\
&&:=\sum^5_{i=1}I_i(T).\label{I3.14}
\end{eqnarray}

For $I_2(T)$, for $t_1,t_2\in [0, T]$ and $x\in\RR^n$, we have
\begin{eqnarray*}
|\bar b(t_1,x)-\bar b(t_2,x)|=\!\!\!\!\!\!\!\!&&\left|\int_{\RR^m}b(t_1,x,z)\mu^{t_1,x}(dz)-\int_{\RR^m}b(t_2,x,z)\mu^{t_2,x}(dz)\right|\nonumber\\
=\!\!\!\!\!\!\!\!&&\left|\int_{\RR^m}b(t_1,x,z)\mu^{t_1,x}(dz)-\tilde{\EE}b(t_1,x, Y^{t_1,x,0}_s)\right|\\
&&+\left|\tilde{\EE}b(t_2,x,Y^{t_2,x,0}_s)-\int_{\RR^m}b(t_2,x,z)\mu^{t_2,x}(dz)\right|\nonumber\\
&&+\left|\tilde{\EE}b(t_1,x,Y^{t_1,x,0}_s)-\tilde{\EE}b(t_2,x,Y^{t_2,x,0}_s)\right|.
\end{eqnarray*}
Then Proposition \ref{Rem 4.1} and Lemma \ref{L3.7} imply that
\begin{eqnarray*}
|\bar b(t_1,x)-\bar b(t_2,x)|\leq\!\!\!\!\!\!\!\!&&C_T e^{-\frac{\beta s}{2}}\left[(K_{t_1}(1))^2+(K_{t_2}(1))^2+|x|^{\theta}+1\right]\\
&&+\tilde{\EE}\left[|Y^{t_1,x,0}_s-Y^{t_2,x,0}_s|(K_{t_1}(1)+|x|^{\theta_3}+|Y^{t_1,x,0}_s|^{\theta_2}+|Y^{t_2,x,0}_s|^{\theta_2})\right]\nonumber\\
&&+|t_1-t_2|^{\gamma_1}\tilde{\EE}\left(|x|^{\theta_3}+|Y^{t_2,x,0}_s|^{\theta_2}+Z_T\right)\\
\leq\!\!\!\!\!\!\!\!&&C_T e^{-\frac{\beta s}{2}}\left[(K_{t_1}(1))^2+(K_{t_2}(1))^2+|x|^{\theta}+1\right]\\
&&+\tilde{\EE}\left[|Y^{t_1,x,0}_s-Y^{t_2,x,0}_s|(K_{t_1}(1)+|x|^{\theta_3}+|Y^{t_1,x,0}_s|^{\theta_2}+|Y^{t_2,x,0}_s|^{\theta_2})\right]\nonumber\\
&&+|t_1-t_2|^{\gamma_1}\tilde{\EE}\left(|x|^{\theta_3}+|Y^{t_2,x,0}_s|^{\theta_2}+Z_T\right)\\
\leq\!\!\!\!\!\!\!\!&&C_T e^{-\frac{\beta s}{2}}\left[(K_{t_1}(1))^2+(K_{t_2}(1))^2+|x|^{\theta}+1\right]\\
&&+C_T |t_1-t_2|^{\gamma_2}\left[(1+|x|^{\frac{2\alpha_2}{\theta_4}\vee \alpha_1})(K_{t_1}(1)+|x|^{\theta_3 \vee \frac{2\theta_2}{\theta_4}})\right]\nonumber\\
&&+C_T|t_1-t_2|^{\gamma_1}\left(|x|^{\theta_3 \vee \frac{2\theta_2}{\theta_4}}+Z_T\right).
\end{eqnarray*}
Then letting $s\rightarrow \infty$, there exits $\tilde{\theta}>0$ such that
\begin{eqnarray*}
|\bar b(t_1,x)-\bar b(t_2,x)|\leq\!\!\!\!\!\!\!\!&&C_T\left[(K_{t_1}(1))^2+|x|^{\tilde{\theta}}+Z_T\right]|t_1-t_2|^{\gamma_1\wedge \gamma_2},
\end{eqnarray*}
which implies that
\begin{eqnarray}
I_2(T)\leq\!\!\!\!\!\!\!\!&&
C\delta^{2(\gamma_1\wedge \gamma_2)}\mathbb{E}\left[\int_{0}^{T\wedge\tilde{\tau}^{\ep}_{R,M}}\left((K_{s}(1))^2+|X^{\ep}_{s(\delta)}|^{\tilde \theta}+Z_T\right)ds\right]^2\nonumber\\
\leq\!\!\!\!\!\!\!\!&&C_{T,R}\delta^{2(\gamma_1\wedge \gamma_2)}.\label{I3.15}
\end{eqnarray}

For $I_3(T)$, note that for any $|x_i|\leq R$, $i=1,2$,
\begin{eqnarray*}
|\bar{b}(t,x_1)-\bar{b}(t,x_2)|\leq \bar K_t(R)|x_1-x_2|^2,
\end{eqnarray*}
where $\bar K_t(R)=C_{t,R}\left[K_t(R)+K_t(1)+1\right]$ (see \eref{4.15} below for a detailed proof). Then we have
\begin{eqnarray}
I_3(T)\leq\!\!\!\!\!\!\!\!&&
\mathbb{E}\left[\int_{0}^{T\wedge\tilde{\tau}^{\ep}_{R,M}}[\bar K_s(R)]^2ds\int^{T\wedge\tilde{\tau}^{\ep}_{R,M}}_0|X^{\ep}_{s(\delta)}-X^{\ep}_{s}|^2ds\right]\nonumber\\
\leq\!\!\!\!\!\!\!\!&&C_{T,R,M}\mathbb{E}\left[\int_{0}^{T\wedge\tilde{\tau}^{\ep}_{R,M}}\big|X^{\ep}_{s(\delta)}-X^{\ep}_{s}\big|^{2}ds\right].\label{I3.16}
\end{eqnarray}

For $I_4(T)$, we have
\begin{eqnarray}
I_4(T)\leq\!\!\!\!\!\!\!\!&&
\mathbb{E}\left[\int_{0}^{T\wedge\tilde{\tau}^{\ep}_{R,M}}[\bar K_s(R)]^2ds\int^{T\wedge\tilde{\tau}^{\ep}_{R,M}}_0|X^{\ep}_{s}-\bar X_{s}|^2ds\right]\nonumber\\
\leq\!\!\!\!\!\!\!\!&&C_{T,R,M}\mathbb{E}\left[\int_{0}^{T\wedge\tilde{\tau}^{\ep}_{R,M}}\big|X^{\ep}_{s}-\bar X_{s}\big|^{2}ds\right]\nonumber\\
\leq\!\!\!\!\!\!\!\!&&C_{T,R,M} \EE\left(\sup_{t\in[0,T\wedge \tilde{\tau}^{\ep}_{R,M}]}|X^{\ep}_t-\hat{X}^{\ep}_{t}|^2\right)+C_{T,R,M} \EE\int^{T\wedge \tilde{\tau}^{\ep}_{R,M}}_0|\hat X^{\ep}_t-\bar{X}_{t}|^2dt.\label{I3.17}
\end{eqnarray}

For $I_5(T)$, it follows that
\begin{eqnarray}
I_5(T)\leq\!\!\!\!\!\!\!\!&&
\mathbb{E}\left\{\int_{0}^{T\wedge\tilde{\tau}^{\ep}_{R,M}}[\bar K_s(R)]^2ds\left[\int^{T\wedge\tilde{\tau}^{\ep}_{R,M}}_{0}|X^{\ep}_s-\bar{X}_{s}|^4ds\right]^{1/2}\right\}\nonumber\\
\leq\!\!\!\!\!\!\!\!&&C_{T,R,M} \EE\left(\sup_{t\in[0,T\wedge \tilde{\tau}^{\ep}_{R,M}]}|X^{\ep}_t-\hat{X}^{\ep}_{t}|^2\right)+\frac{1}{2}\EE\left(\sup_{t\in[0,T\wedge \tilde{\tau}^{\ep}_{R,M}]}|\hat{X}_{t}^{\ep}-\bar{X}_{t}|^2\right)\nonumber\\
&&+C_{T,R,M} \EE\int^{T\wedge \tilde{\tau}^{\ep}_{R,M}}_0|\hat X^{\ep}_t-\bar{X}_{t}|^2dt.\label{I3.18}
\end{eqnarray}
By \eref{I3.14}-\eref{I3.18}, we get
\begin{eqnarray}
&&\EE\left(\sup_{t\in[0,T\wedge \tilde{\tau}^{\ep}_{R,M}]}|\hat{X}_{t}^{\ep}-\bar{X}_{t}|^2\right)\nonumber\\
\leq\!\!\!\!\!\!\!\!&&C_{T,R,M} \EE\left(\sup_{t\in[0,T\wedge \tilde{\tau}^{\ep}_{R,M}]}|X^{\ep}_t-\hat{X}^{\ep}_{t}|^2\right)+C_{T,R,M}\mathbb{E}\left[\int_{0}^{T\wedge\tilde{\tau}^{\ep}_{R,M}}\big|X^{\ep}_{s(\delta)}-X^{\ep}_{s}\big|^{2}ds\right]\nonumber\\
&&+C_{T,R,M}\EE \int^{T\wedge \tilde{\tau}^{\ep}_{R,M}}_0|\hat X^{\ep}_t-\bar{X}_{t}|^2dt+C_{T,R,M}\delta^{2(\gamma_1\wedge \gamma_2)}+I_1(T).\label{I3.19}
\end{eqnarray}

Next, we intend to estimate the term $I_1(T)$. Note that
\begin{eqnarray}    \label{J2}
&&\left|\int_{0}^{t}\left[b(s(\delta), X_{s(\delta)}^{\ep},\hat{Y}_{s}^{\ep})-\bar{b}(s(\delta), X^{\ep}_{s(\delta)})\right]ds\right|^2\nonumber\\
=\!\!\!\!\!\!\!\!&&\left|\sum_{k=0}^{[t/\delta]-1}
\int_{k\delta}^{(k+1)\delta}\left[b(k\delta, X_{k\delta}^{\ep},\hat{Y}_{s}^{\ep})-\bar{b}(k\delta, X_{k\delta}^{\ep})\right]ds
\!\!+\!\!\int_{t(\delta)}^{t}\left[b(t(\delta), X_{t(\delta)}^{\ep},\hat{Y}_{s}^{\ep})-\bar{b}(t(\delta), X_{t(\delta)}^{\ep})\right]ds\right|^2\nonumber\\
\leq\!\!\!\!\!\!\!\!&&2[t/\delta]\sum_{k=0}^{[t/\delta]-1}
\left|\int_{k\delta}^{(k+1)\delta}\left[b(k\delta, X_{k\delta}^{\ep},\hat{Y}_{s}^{\ep})-\bar{b}(k\delta, X_{k\delta}^{\ep})\right]ds\right|^2\nonumber\\
&&+2\left|\int_{t(\delta)}^{t}\left[b(t(\delta),X_{t(\delta)}^{\ep},\hat{Y}_{s}^{\ep})-\bar{b}(t(\delta),X_{t(\delta)}^{\ep})\right]ds\right|^2\nonumber\\
:=\!\!\!\!\!\!\!\!&&I_{11}(t)+I_{12}(t).
\end{eqnarray}
For $I_{12}(t)$, by Lemma \ref{MDY}, we easily deduce that
\begin{eqnarray}
\EE\left[\sup_{t\in [0, T\wedge\tilde{\tau}^{\ep}_{R,M}]}I_{12}(t)\right]\leq\!\!\!\!\!\!\!\!&&\delta\EE\left[\sup_{t\in [0, T\wedge\tilde{\tau}^{\ep}_{R,M}]}\int_{t(\delta)}^{t}[K_{t(\delta)}(1)]^2+|X^{\ep}_{t(\delta)}|^{2\theta_5}+|\hat{Y}_{s}^{\ep}|^{2\theta_6})ds\right]\nonumber\\
\leq\!\!\!\!\!\!\!\!&&\delta\left[\sup_{t\in[0,T]}\EE[K_{t}(1)]^2+R^{2\theta_5}+\int^T_0\EE|\hat{Y}_{s}^{\ep}|^{2\theta_6}ds\right]\nonumber\\
\leq\!\!\!\!\!\!\!\!&&C_{T,R,M}(|x|^{\frac{4\theta_6}{\theta_4}}+|y|^{2\theta_6}+1)\delta.
\end{eqnarray}
Now, we estimate the term  $I_{11}(t)$,
\begin{eqnarray*}
&&\mathbb{E}\left[\sup_{t\in[0, T\wedge\tilde{\tau}^{\ep}_{R,M}]}I_{11}(t)\right]\nonumber\\
\leq\!\!\!\!\!\!\!\!&&C[T/\delta]\mathbb{E}\sum_{k=0}^{[T/\delta]-1}
\left[\left|\int_{k\delta}^{(k+1)\delta}\left[b(k\delta, X_{k\delta}^{\ep},\hat{Y}_{s}^{\ep})-\bar{b}(k\delta,X_{k\delta}^{\ep})\right]ds\right|^{2}1_{\{k\delta\leq \tilde{\tau}^{\ep}_{R,M}\}}\right]\nonumber\\
\leq\!\!\!\!\!\!\!\!&&\frac{C_{T}}{\delta^{2}}\max_{0\leq k\leq[T/\delta]-1}\mathbb{E}\left[\left|\int_{k\delta}^{(k+1)\delta}
\left[b(k\delta, X_{k\delta}^{\ep},\hat{Y}_{s}^{\ep})-\bar{b}(k\delta, X_{k\delta}^{\ep})\right]ds\right|^{2} 1_{\{k\delta\leq \tilde{\tau}^{\ep}_{R,M}\}}\right] \nonumber\\
=\!\!\!\!\!\!\!\!&&C_{T}\frac{\ep^{2}}{\delta^{2}}\max_{0\leq k\leq[T/\delta]-1}\mathbb{E}\left[\left|\int_{0}^{\frac{\delta}{\ep}}
\left[b(k\delta, X_{k\delta}^{\ep},\hat{Y}_{s\ep+k\delta}^{\ep})-\bar{b}(k\delta, X_{k\delta}^{\ep})\right]ds\right|^{2}1_{\{k\delta\leq \tilde{\tau}^{\ep}_{R,M}\}}\right]  \nonumber\\
=\!\!\!\!\!\!\!\!&&C_{T}\frac{\ep^{2}}{\delta^{2}}\max_{0\leq k\leq[T/\delta]-1}\int_{0}^{\frac{\delta}{\ep}}
\int_{r}^{\frac{\delta}{\ep}}\Psi_{k}(s,r)dsdr,  \nonumber
\end{eqnarray*}
where for any $0\leq r\leq s\leq \frac{\delta}{\ep}$,
\begin{eqnarray*}
\Psi_{k}(s,r):=\!\!\!\!\!\!\!\!&&\mathbb{E}\left[
\langle b(k\delta, X_{k\delta}^{\ep},\hat{Y}_{s\ep+k\delta}^{\ep})-\bar{b}(k\delta, X_{k\delta}^{\ep}),
b(k\delta, X_{k\delta}^{\ep},\hat{Y}_{r\ep+k\delta}^{\ep})-\bar{b}(k\delta, X_{k\delta}^{\ep})\rangle1_{\{k\delta\leq \tilde{\tau}^{\ep}_{R,M}\}}\right].
\end{eqnarray*}
For any $\ep, s>0$, and $\mathscr{F}_s$-measurable $\RR^n$- resp. $\RR^m$-valued maps $X$ and $Y$, we consider the following equation
\begin{eqnarray*}
\tilde{Y}^{\ep,s,X,Y}_t=Y+\frac{1}{\ep}\int^t_s f(s,X,\tilde{Y}^{\ep,s,X,Y}_r)dr+\frac{1}{\sqrt{\ep}}\int^t_s g(s,X,\tilde{Y}^{\ep,s,X,Y}_r)dW^2_r,\quad t\geq s.
\end{eqnarray*}
Then by the construction of $\hat{Y}_{t}^{\ep}$,
for any $k\in \mathbb{N}_{\ast}$ and $t\in[k\delta,(k+1)\delta]$ we have
$$
\hat{Y}_{t}^{\ep}=\tilde Y^{\ep,k\delta,X_{k\delta}^{\ep},\hat{Y}_{k\delta}^{\ep}}_t,
$$
which implies that
\begin{eqnarray*}
\Psi_{k}(s,r)=\!\!\!\!\!\!\!\!&&\mathbb{E}\left[
\big\langle b\left(k\delta, X_{k\delta}^{\ep},\tilde {Y}^{\ep, k\delta, X_{k\delta}^{\ep}, \hat Y_{k\delta}^{\ep}}_{s\ep+k\delta}\right)-\bar{b}(k\delta, X_{k\delta}^{\ep}),\right.\\
&&\quad\quad\quad \left.b\left(k\delta, X_{k\delta}^{\ep},\tilde{Y}^{\ep, k\delta, X_{k\delta}^{\ep}, \hat Y_{k\delta}^{\ep}}_{r\ep+k\delta}\right)-\bar{b}(k\delta, X_{k\delta}^{\ep})\big\rangle 1_{\{k\delta\leq \tilde{\tau}^{\ep}_{R,M}\}}\right].
\end{eqnarray*}

By approximating by functions of type $(x,y)\rightarrow H_1(x)H_2(y)$, one sees that for any measurable functions $H:\RR^n\times\RR^m\rightarrow \RR^{m\times d_2}$, $\phi: \RR^m\rightarrow\RR^n$, and for any $\mathscr{F}_{s}$-measurable $\RR^n$-valued map $X$ and $\mathscr{F}_{t}$-adapted $\RR^m$-valued process $\{Z_t\}_{t\geq s}$, we have for any $t>s$,
\begin{eqnarray}
\EE\left[\phi\left(\int^t_s H(X, Z_r)dW^2_r\right)|\mathscr{F}_s\right](\omega)=\EE\left[\phi\left(\int^t_s H(X(\omega), Z_r)dW^2_r\right)|\mathscr{F}_s\right](\omega),~ \PP-a.s..\label{MP}
\end{eqnarray}
Note that for any fixed $(x,y)\in\RR^n\times\RR^m$, $X_{k\delta}^{\ep}$, $\hat Y_{k\delta}^{\ep}$ , $b(k\delta, x, y)$ and $1_{\{k\delta\leq \tilde{\tau}^{\ep}_{R,M}\}}$ are $\mathscr{F}_{k\delta}$-measurable, $\{\tilde{Y}^{\ep, k\delta, x, y}_{s\ep+k\delta}\}_{s\geq 0}$ is independent of $\mathscr{F}_{k\delta}$, and by statement \eref{MP}, we have

\begin{eqnarray}
\Psi_{k}(s,r)=\!\!\!\!\!\!\!\!&&\int_{\Omega}\mathbb{E}\left[
\big\langle b\left(k\delta, X_{k\delta}^{\ep},\tilde {Y}^{\ep, k\delta, X_{k\delta}^{\ep}, \hat Y_{k\delta}^{\ep}}_{s\ep+k\delta}\right)-\bar{b}(k\delta, X_{k\delta}^{\ep}),\right.\nonumber\\
&&\quad\quad\quad \left.b\left(k\delta, X_{k\delta}^{\ep},\tilde{Y}^{\ep, k\delta, X_{k\delta}^{\ep}, \hat Y_{k\delta}^{\ep}}_{r\ep+k\delta}\right)-\bar{b}(k\delta, X_{k\delta}^{\ep})\big\rangle 1_{\{k\delta\leq \tilde{\tau}^{\ep}_{R,M}\}}|\mathscr{F}_{k\delta}\right](\omega)\PP(d\omega)\nonumber\\
=\!\!\!\!\!\!\!\!&&\int_{\Omega}\Big[\mathbb{E}
\big\langle b\left(k\delta, X_{k\delta}^{\ep}(\omega),\tilde {Y}^{\ep, k\delta, X_{k\delta}^{\ep}(\omega), \hat Y_{k\delta}^{\ep}(\omega)}_{s\ep+k\delta}\right)-\bar{b}(k\delta, X_{k\delta}^{\ep}(\omega)),\nonumber\\
&& b\left(k\delta, X_{k\delta}^{\ep}(\omega),\tilde{Y}^{\ep, k\delta,X_{k\delta}^{\ep}(\omega), \hat Y_{k\delta}^{\ep}(\omega)}_{r\ep+k\delta}\right)-\bar{b}(k\delta, X_{k\delta}^{\ep}(\omega))\big\rangle 1_{\{k\delta\leq \tilde{\tau}^{\ep}_{R,M}\}}(\omega)\Big]\PP(d\omega).\label{3.18'}
\end{eqnarray}
For any given $x\in\RR^n, y\in\RR^m$, by the definition of process $\tilde{Y}^{\ep,s,x,y}_t$, it is easy to see
\begin{eqnarray}
\tilde{Y}^{\ep,k\delta,x,y}_{s\ep+k\delta}
=\!\!\!\!\!\!\!\!&&y+\frac{1}{\ep}\int^{s\ep}_{0} f(k\delta,x,\tilde{Y}^{\ep,k\delta,x,y}_{r+k\delta})dr+\frac{1}{\sqrt{\ep}}\int^{s\ep}_{0} g(k\delta,x,\tilde{Y}^{\ep,k\delta,x,y}_{r+k\delta})dW^{2,k\delta}_r\nonumber\\
=\!\!\!\!\!\!\!\!&&y+\int^{s}_{0} f(k\delta,x,\tilde{Y}^{\ep,k\delta,x,y}_{r\ep+k\delta})dr+\int^{s}_{0} g(k\delta,x,\tilde{Y}^{\ep,k\delta,x,y}_{r\ep+k\delta})d\hat{W}^{2,k\delta}_r,\label{E3.12.1}
\end{eqnarray}
where $\{W^{2, k\delta}_r:=W^2_{r+k\delta}-W^2_{k\delta}\}_{r\geq 0}$ and $\{\hat W^{2,k\delta}_t:=\frac{1}{\sqrt{\ep}}W^{2,k\delta}_{t\ep}\}_{t\geq 0}$. Recall the solution of the frozen equation satisfies
\begin{eqnarray}
Y_{s}^{k\delta, x, y}=\!\!\!\!\!\!\!\!&& y
+\int_{0}^{{s}}f(k\delta, x, Y_{r}^{k\delta, x,y})dr
+\int_{0}^{{s}}g(k\delta, x, Y_{r}^{k\delta, x,y})d\tilde{W}^2_r.  \label{E3.12.2}
\end{eqnarray}
The uniqueness of solutions of Eq. (\ref{E3.12.1}) and Eq. (\ref{E3.12.2}) implies
that the distribution of $(\tilde Y^{\ep, k\delta, x,y}_{s\ep+k\delta})_{0\leq s\leq \delta/\ep}$
coincides with the distribution of $(Y_{s}^{k\delta,x, y})_{0\leq s\leq \delta/\ep}$.

By a similar argument in Proposition  \ref{Rem 4.1} and condition (ii), we can obtain
\begin{eqnarray}
&&|b(k\delta, x,y)-\bar{b}(k\delta, x)|\nonumber\\
=\!\!\!\!\!\!\!\!&& \left|\int_{\RR^m}b(k\delta, x,y)-\tilde \EE b(k\delta, x, Y^{k\delta,x,z}_s)\mu^{k\delta,x}(dz)\right|\nonumber\\
\leq\!\!\!\!\!\!\!\!&& C_T\int_{\RR^m}\tilde \EE\left[\left| y-Y^{k\delta,x,z}_s\right|(|y|^{\theta_2}+|Y^{k\delta,x,z}_s|^{\theta_2}+|x|^{\theta_3}+K_{k\delta}(1))\right]\mu^{k\delta,x}(dz)\nonumber\\
\leq\!\!\!\!\!\!\!\!&& C_T\int_{\RR^m}\left[\tilde \EE\left(\left|y-Y^{k\delta,x,z}_s\right|^2\right)\right]^{1/2}\left[\tilde \EE\left(|Y^{k\delta,x,y}_s|^{2\theta_2}+|Y^{k\delta,x,z}_s|^{2\theta_2}+|x|^{2\theta_3}+[K_t(1)]^2\right)\right]^{1/2}\mu^{k\delta,x}(dz)\nonumber\\
\leq\!\!\!\!\!\!\!\!&&C_T\int_{\RR^m}(|y|+|z|+|x|^{2/\theta_4})\left[|z|^{\theta_2}+|y|^{\theta_2}+|x|^{\frac{2\theta_2}{\theta_4}}+|x|^{\theta_3}+K_t(1)+1\right]\mu^{k\delta,x}(dz)\nonumber\\
\leq\!\!\!\!\!\!\!\!&&C_T\left[(K_{k\delta}(1))^2+1+|x|^{\frac{4\theta_2}{\theta_4}\vee (2\theta_3)}+|y|^{\theta_2+1}\right].\label{3.21'}
\end{eqnarray}
Then by  \eref{3.18'}, \eref{3.21'} and Proposition \ref{Rem 4.1}, we have
\begin{eqnarray*}
\Psi_{k}(s,r)=\!\!\!\!\!\!\!\!&&\int_{\Omega}\Big[\tilde{\mathbb{E}}
\big\langle b\left(k\delta, X_{k\delta}^{\ep}(\omega),Y^{k\delta, X_{k\delta}^{\ep}(\omega), \hat Y_{k\delta}^{\ep}(\omega)}_{s}\right)-\bar{b}(k\delta, X_{k\delta}^{\ep}(\omega)),\nonumber\\
&&\quad\quad\quad b\left(k\delta, X_{k\delta}^{\ep}(\omega),Y^{k\delta, X_{k\delta}^{\ep}(\omega), \hat Y_{k\delta}^{\ep}(\omega)}_{r}\right)-\bar{b}(k\delta, X_{k\delta}^{\ep}(\omega))\big\rangle 1_{\{k\delta\leq \tilde{\tau}^{\ep}_{R,M}\}}(\omega)\Big]\PP(d\omega)\nonumber\\
=\!\!\!\!\!\!\!\!&&\int_{\Omega}\int_{\tilde{\Omega}}\big\langle\tilde{\mathbb{E}}\Big[
 b\left(k\delta, X_{k\delta}^{\ep}(\omega),Y^{k\delta, X_{k\delta}^{\ep}(\omega),Y_{r}^{k\delta,X_{k\delta}^{\ep}(\omega),\hat Y_{k\delta}^{\ep}(\omega)}(\tilde{\omega})}_{s-r}\right)-\bar{b}(k\delta, X_{k\delta}^{\ep}(\omega))\Big],\nonumber\\
&&\quad\quad\quad b\left(k\delta, X_{k\delta}^{\ep}(\omega),Y^{k\delta, X_{k\delta}^{\ep}(\omega), \hat Y_{k\delta}^{\ep}(\omega)}_{r}(\tilde{\omega})\right)-\bar{b}(k\delta, X_{k\delta}^{\ep}(\omega))\big\rangle 1_{\{k\delta\leq \tilde{\tau}^{\ep}_{R,M}\}}(\omega)\tilde{\PP}(d\tilde{\omega})\PP(d\omega)\nonumber\\
\leq\!\!\!\!\!\!\!\!&&\int_{\Omega}\int_{\tilde{\Omega}}\left[(K_{k\delta}(1))^2\!+\!1\!+\!|X_{k\delta}^{\ep}(\omega)|^{(\frac{2\theta_2+2}{\theta_4})\vee(\frac{\theta_3\theta_4+2}{\theta_4})\vee(\frac{\theta_3(\theta_2+1)}{\theta_2})}+|Y_{r}^{k\delta, X_{k\delta}^{\ep}(\omega), \hat Y_{k\delta}^{\ep}(\omega)}(\tilde{\omega})|^{\theta_2+1}\right]e^{-\frac{(s-r)\beta}{2}}\nonumber\\
&&\quad\cdot\left[(K_{k\delta}(1))^2\!+\!1\!+\!|X_{k\delta}^{\ep}(\omega)|^{\frac{4\theta_2}{\theta_4}\vee (2\theta_3)}+|Y_{r}^{k\delta,X_{k\delta}^{\ep}(\omega), \hat Y_{k\delta}^{\ep}(\omega)}(\tilde{\omega})|^{\theta_2+1}\right]1_{\{k\delta\leq \tilde{\tau}^{\ep}_{R,M}\}}(\omega)\tilde{\PP}(d\tilde{\omega})\PP(d\omega)\nonumber\\
\leq\!\!\!\!\!\!\!\!&&C_T\int_{\Omega}\left[(K_{k\delta}(1))^4+|X^{\ep}_{k\delta}(\omega)|^{\frac{8\theta_2}{\theta_4}\vee (4\theta_3)}+|\hat Y_{k\delta}^{\ep}(\omega)|^{2\theta_2+2}+1)1_{\{k\delta\leq \tilde{\tau}^{\ep}_{R,M}\}}(\omega)\right]\PP(d\omega)e^{-\frac{(s-r)\beta}{2}}\\
\leq\!\!\!\!\!\!\!\!&&C_{T,R}(|x|^{\frac{4(\theta_2+1)}{\theta_4}}+|y|^{2(\theta_2+1)}+1)e^{-\frac{(s-r)\beta}{2}},
\end{eqnarray*}
where the last inequality comes from the definition of stopping time, Lemmas \ref{PMY} and \ref{MDY}.

Hence we have
\begin{eqnarray}
\mathbb{E}\left[\sup_{t\in [0,T\wedge\tilde{\tau}^{\ep}_{R,M}]}I_1(t)\right]\leq\!\!\!\!\!\!\!\!&&C_{T,R,x,y}\frac{\ep^{2}}{\delta^{2}}
\int_{0}^{\frac{\delta}{\ep}}\int_{r}^{\frac{\delta}{\ep}}e^{-\frac{(s-r)\beta}{2}}dsdr +C_{T,R,M,x,y}\delta  \nonumber\\
=\!\!\!\!\!\!\!\!&&C_{T,R,x,y}\frac{\ep^{2}}{\delta^{2}}\Big(\frac{\delta}{\beta\ep}-\frac{1}{\beta^{2}}
+\frac{1}{\beta^2}e^{-\frac{\beta\delta}{\ep}}\Big) +C_{T,R,M,x,y}\delta  \nonumber\\
\leq\!\!\!\!\!\!\!\!&&C_{T,R,x,y}\frac{\ep}{\delta}+C_{T,R,M,x,y}\delta.\label{3.15}
\end{eqnarray}

According to estimates $(\ref{I3.19})$ and $(\ref{3.15})$, we obtain that
\begin{eqnarray*}
\mathbb{E}\left(\sup_{t\in[0,T\wedge\tilde{\tau}^{\ep}_{R,M}]}|\hat{X}^{\ep}_{t}-\bar{X}_{t}|^{2}\right)
\leq\!\!\!\!\!\!\!\!&&C_{T,R,M,x,y}\left(\frac{\ep}{\delta}+\delta^{\gamma}\right)\\
&&+C_{T,R,M}\int_{0}^{T}\EE\left(\sup_{s\in[0,t\wedge\tilde{\tau}^{\ep}_{R,M}]}
\big|\hat{X}^{\ep}_{s}-\bar{X}_{s}\big|^{2}\right)dt,
\end{eqnarray*}
where $\gamma=\min\{2\gamma_1, \gamma_2, 1/2\}$. By Gronwall's inequality, we get
\begin{eqnarray*}
\mathbb{E}\Big(\sup_{t\in [0, T\wedge\tilde{\tau}^{\ep}_{R,M}]}|\hat{X}_{t}^{\ep}-\bar{X}_{t}|^{2}\Big)\leq\!\!\!\!\!\!\!\!&& C_{T,R,M,x,y}\left(\frac{\ep}{\delta}+\delta^{\gamma}\right).
\end{eqnarray*}
Hence the proof is complete.
\end{proof}

\vspace{0.2cm}
Now we can finish the proof of  our main result.

\vspace{0.2cm}

\noindent\textbf{Proof of Theorem \ref{main result 1}} Taking $\delta=\ep^{\tilde\gamma}$ with $\tilde\gamma=(1+\min\{2\gamma_1, \gamma_2, 1/2\})^{-1}$, Lemmas $\ref{DEX}$ and $\ref{ESX}$ imply that
\begin{eqnarray}
\mathbb{E}\left(\sup_{t\in [0, T\wedge\tilde{\tau}^{\ep}_{R,M}]}|X^{\ep}_{t}-\bar{X}_{t}|\right)
\leq\!\!\!\!\!\!\!\!&&\mathbb{E}\left(\sup_{t\in [0, T\wedge\tilde{\tau}^{\ep}_{R,M}]}|X^{\ep}_{t}-\hat{X}^{\ep}_{t}|
+\sup_{t\in [0, T\wedge\tilde{\tau}^{\ep}_{R,M}]}|\hat X^{\ep}_{t}-\bar{X}_{t}|\right)\nonumber\\
\leq\!\!\!\!\!\!\!\!&&C_{T,R,M,x,y}\left(\sqrt{{\ep}{\delta}^{-1}}+\delta^{\frac{1}{2}\min\{2\gamma_1, \gamma_2, 1/2\}}\right)\nonumber\\
\leq\!\!\!\!\!\!\!\!&&C_{T,R,M,x,y}\ep^{\frac{1-\tilde\gamma}{2}}.\label{before tau_R}
\end{eqnarray}
By Chebyshev's inequality, Lemmas \ref{MDY} and \ref{PMA}, we have
\begin{eqnarray}
&&\mathbb{E}\left(\sup_{t\in [0, T]}|X_{t}^{\ep}-\bar{X}_{t}| 1_{\{T>\tilde{\tau}^{\ep}_{R,M}\}}\right)\nonumber\\
\leq\!\!\!\!\!\!\!\!&&\left[\mathbb{E}\left(\sup_{t\in [0, T]}|X_{t}^{\ep}-\bar{X}_{t}|^{2}\right)\right]^{\frac{1}{2}}
\cdot\left[\mathbb{P}(T>\tilde{\tau}^{\ep}_{R,M})\right]^{\frac12} \nonumber\\
\leq\!\!\!\!\!\!\!\!&& \frac{C_{T,x,y}}{R^{1/2}}\left[\mathbb{E}\left(\sup_{t\in [0, T]}|X_{t}^{\ep}|+\sup_{t\in [0, T]}|\bar{X}_{t}|
+\int^T_0|Y^{\ep}_s|^{2\theta_2}ds+\int^T_0|\hat Y^{\ep}_s|^{4\theta_1\vee 2\theta_2}ds\right)\right]^{1/2}\nonumber\\
&&+ \frac{C_{T,x,y}}{M^{1/2}}\left[\mathbb{E}\int^T_0 [K_s(R)]^4ds\right]^{1/2}
\nonumber\\
\leq\!\!\!\!\!\!\!\!&&\frac{C_{T,x,y}}{R^{1/2}}+\frac{C_{T,R,x,y}}{M^{1/2}}. \label{after tau_R}
\end{eqnarray}
Hence, by \eref{before tau_R} and \eref{after tau_R}, we obtain that
\begin{eqnarray*}
\mathbb{E}\left(\sup_{t\in [0, T]}|X_{t}^{\ep}-\bar{X}_{t}|\right)\leq\!\!\!\!\!\!\!\!&& C_{T,R,M,x,y}\ep^{\frac{1-\tilde \gamma}{2}}+ \frac{C_{T,x,y}}{R^{1/2}}+\frac{C_{T,R,x,y}}{M^{1/2}}.
\end{eqnarray*}
Now,  letting $\ep\rightarrow 0$ firstly, $M\rightarrow \infty$ secondly, and $R\rightarrow \infty$ finally, we obtain that
\begin{eqnarray}
\lim_{\ep\rightarrow 0}\mathbb{E}\left(\sup_{t\in [0, T]}|X_{t}^{\ep}-\bar{X}_{t}|\right)=0.\label{L1}
\end{eqnarray}
On the  one hand, if $\lambda_1=0$ in \textbf{$({\bf H}_1)$} and \textbf{$({\bf A}_{\tilde \theta_1})$} holds with $\tilde \theta_1=\max\{4\theta_1, 2\theta_2+2, 2\theta_6, 4\alpha_2\}$, then by Lemma \ref{PMY} and \ref{PMA} we have
$$
\EE\left(\sup_{t\in [0, T]}|X_{t}^{\ep}-\bar{X}_{t}|^p\right)
\leq C_p\EE\left(\sup_{t\in [0, T]}|X_{t}^{\ep}|^p+\sup_{t\in [0, T]}|\bar{X}_{t}|^p\right)<\infty,\quad \forall p>0.
$$
On the other hand, if $\lambda_1>0$ in \textbf{$({\bf H}_1)$} and \textbf{$({\bf A}_{k})$} with some $k>\tilde \theta_2$ holds, where $\tilde \theta_2=\max\{4\theta_1, (2\theta_2+2), 2\theta_6, 4\alpha_2, \theta_5\theta_4, 2\alpha_1\theta_4\}$. By Lemmas \ref{PMY} and \ref{PMA}, for any $k'<k$ we have
$$
\EE\left(\sup_{t\in [0, T]}|X_{t}^{\ep}-\bar{X}_{t}|^{2k'/\theta_4}\right)
\leq C_k\EE\left(\sup_{t\in [0, T]}|X_{t}^{\ep}|^{2k'/\theta_4}+\sup_{t\in [0, T]}|\bar{X}_{t}|^{2k'/\theta_4}\right)<\infty.
$$
Hence by H\"{o}lder's inequality and \eref{L1}, it is easy to prove that \eref{R1} and \eref{R2} hold.
Therefore, the proof is complete.  \qed

\section{Examples}

In this section, we give several concrete examples to illustrate the applicability of our main result. We concentrate on cases which are not covered by previous papers in the literature. For simplicity, we only consider the 1-dimensional case, but one can
easily extend to the multi-dimensional case.

\begin{example} Let us consider the following slow-fast SDEs,
\begin{equation}\left\{\begin{array}{l}\label{EX1}
\displaystyle
d X^{\ep}_t = \left[-(X^{\ep}_t)^3+X^{\ep}_t+(Y^{\ep}_t)^3\right]dt+X^{\ep}_td W^{1}_t,\quad X^{\ep}_0=x\in \RR,\\
\displaystyle
d Y^{\ep}_t =\frac{1}{\ep}\left[-(X^{\ep}_t)^2 (Y^{\ep}_t)^3- 3Y^{\ep}_t-(Y^{\ep}_t)^5\right]dt+\frac{1}{\sqrt{\ep}}\left[\sin(X^{\ep}_t)+\sin(Y^{\ep}_t)\right]d W^{2}_t,\quad Y^{\ep}_0=y\in \RR,
\end{array}\right.
\end{equation}
where $\{W^{1}_t\}_{t\geq 0}$ and $\{W^{2}_t\}_{t\geq 0}$ are independent $1$-dimensional Brownian motions.

Let
$$
b(x,y)=-x^3+x+y^3, \quad \sigma(x)=x
$$
and
$$
f(x,y)=-x^2 y^3-3y-y^5, \quad g(x,y)=\sin x+\sin y.
$$
It is easy to verify that \textbf{$({\bf H}_1)$} with $\theta_4=6$, \textbf{$({\bf H}_2)$} and \textbf{$({\bf A}_k)$} with any $k\geq 2$ hold.

Hence,  by Theorem \ref{main result 1} for any $p>0$ we have
\begin{eqnarray*}
\lim_{\ep\rightarrow 0}\mathbb{E}\left(\sup_{t\in [0, T]}|X_{t}^{\ep}-\bar{X}_{t}|^p\right)=0,
\end{eqnarray*}
where $\bar{X}_{t}$ is the solution of the corresponding averaged equation.
\end{example}

\vspace{0.2cm}

\begin{example} Let us consider the following slow-fast SDEs,
\begin{equation}\left\{\begin{array}{l}\label{EX2}
\displaystyle
d X^{\ep}_t = \left[t^2X^{\ep}_t-(X^{\ep}_t)^3 (Y^{\ep}_t)^2+\lambda_1 Y^{\ep}_t\right]dt+\left(t^2+X^{\ep}_t\right)d W^{1}_t,\quad X^{\ep}_0=x\in \RR,\\
\displaystyle
d Y^{\ep}_t =\frac{1}{\ep}\left(\sqrt{t}X^{\ep}_t-8Y^{\ep}_t\right)dt+\frac{1}{\sqrt{\ep}}(t+X^{\ep}_t+Y^{\ep}_t)d W^{2}_t,\quad Y^{\ep}_0=y\in \RR,
\end{array}\right.
\end{equation}
where $\lambda_1\geq0$, $\{W^{1}_t\}_{t\geq 0}$ and $\{W^{2}_t\}_{t\geq 0}$ are independent $1$-dimensional Brownian motion.

Let
$$
b(t, x,y)=t^2x-x^3 y^2+\lambda_1 y, \quad \sigma(t, x)=t^2+x
$$
and
$$
f(t, x,y)=\sqrt{t}x-8y, \quad g(t,x,y)=t+x+y.
$$
It is easy to verify that \textbf{$({\bf H}_1)$} holds with $\theta_1=2,\theta_2=1,\theta_3=3,\theta_4=2,\theta_5=6,\theta_6=4$, $\gamma_1=1$, $Z_T=0$ and $K_t(R)=6R^2+2 t^4+2$; \textbf{$({\bf H}_2)$} holds with $\alpha_i=1, i=1,2,3,4$, and $\gamma_2=1/2$; \textbf{$({\bf A}_k)$} holds with any $2\leq k<17$.

Hence,  by Theorem \ref{main result 1}, if $\lambda_1=0$, for any $p>0$ we have
\begin{eqnarray*}
\lim_{\ep\rightarrow 0}\mathbb{E}\left(\sup_{t\in [0, T]}|X_{t}^{\ep}-\bar{X}_{t}|^p\right)=0.
\end{eqnarray*}
Moreover, if $\lambda_1>0$, for any $0<p<17$ we have
\begin{eqnarray*}
\lim_{\ep\rightarrow 0}\mathbb{E}\left(\sup_{t\in [0, T]}|X_{t}^{\ep}-\bar{X}_{t}|^p\right)=0,
\end{eqnarray*}
where $\bar{X}_{t}$ is the solution of the corresponding averaged equation.
\end{example}

\vspace{0.2cm}

\begin{example} Let us consider the following slow-fast SDEs,
\begin{equation}\left\{\begin{array}{l}\label{EX3}
\displaystyle
d X^{\ep}_t = \left[ -|\sin(W^1_t)|(X^{\ep}_t)^3+Y^{\ep}_t\right]dt+X^{\ep}_td W^{1}_t,\quad X^{\ep}_0=x\in \RR,\\
\displaystyle
d Y^{\ep}_t =\frac{1}{\ep}\left[X^{\ep}_t-8Y^{\ep}_t\right]dt+\frac{1}{\sqrt{\ep}}Y^{\ep}_td W^{2}_t,\quad Y^{\ep}_0=y\in \RR,
\end{array}\right.
\end{equation}
where $\{W^{1}_t\}_{t\geq 0}$ and $\{W^{2}_t\}_{t\geq 0}$ are independent $1$-dimensional Brownian motions.

Assume that
$$
b(t,x,y,\omega)=-|\sin(W^1_t(\omega))|x^3+ y, \quad \sigma(x)=x
$$
and
$$
f(x,y)=x-8y, \quad g(x,y)=y.
$$
It is easy to verify that \textbf{$({\bf H}_1)$} holds with $\theta_1=0,\theta_2=1,\theta_3=3,\theta_4=2,\theta_5=3,\theta_6=1$, $\gamma_1<1/2$, $K_t(R)=6R^2+1$ and $Z_T=\sup_{0\leq s<t\leq T}\frac{|W^1_t-W^1_s|}{|t-s|^{\gamma_1}}$ with $\EE Z^2_T<\infty$ by Kolmogorov's continuity theorem ; \textbf{$({\bf H}_2)$} holds with $\alpha_i=1, i=1,2,3,4$; \textbf{$({\bf A}_k)$} holds with any $2\leq k<17$.

Hence, by Theorem \ref{main result 1}, for any $0<p<17$ we have
\begin{eqnarray*}
\lim_{\ep\rightarrow 0}\mathbb{E}\left(\sup_{t\in [0, T]}|X_{t}^{\ep}-\bar{X}_{t}|^p\right)=0,
\end{eqnarray*}
where $\bar{X}_{t}$ is the solution of corresponding averaged equation.
\end{example}

\section{Appendix}

In this section, using the classical result of Krylov (cf. \cite[Theorem 3.1.1]{LR}), we prove the existence and uniqueness of solutions to system (\ref{Equation}) and the corresponding averaged equation.

\subsection{Proof of Theorem \ref{main}}
\begin{proof}
We denote
$$
Z^{\ep}_t=\left(
                               \begin{array}{c}
                                 X^{\ep}_t \\
                                 Y^{\ep}_t \\
                               \end{array}
                             \right)
,\quad \tilde{b}^{\ep}(t,x,y)=\left(
                             \begin{array}{c}
                               b(t,x,y) \\
                               \frac{1}{\ep}f(t,x,y) \\
                             \end{array}
                           \right)
$$
and
$$
\tilde{\sigma}^{\ep}(t,x,y)=\left(
                                  \begin{array}{cc}
                                    \sigma(t,x) & 0 \\
                                    0 & \frac{1}{\sqrt{\ep}}g(t,x,y) \\
                                  \end{array}
                                \right)
, \quad W_t=\left(
             \begin{array}{c}
               W^1_t \\
               W^2_t \\
             \end{array}
           \right).
$$
Then system \eref{Equation} can be rewritten as the following equation
\begin{equation}
 dZ^{\ep}_t=\tilde{b}^{\ep}(t, Z^{\ep}_t)dt+\tilde{\sigma}^{\ep}(t, Z^{\ep}_t)d W_t,\quad
Z^{\ep}_{0}=\left(
                      \begin{array}{c}
                        x \\
                        y \\
                      \end{array}
                    \right)
.\label{Eq2}
\end{equation}
Under the assumptions \textbf{$({\bf H}_1)$} and \textbf{$({\bf H}_2)$}, we intend to prove the coefficients in Eq. \eref{Eq2} satisfy the local weak monotonicity and weak coercivity conditions in \cite[Theorem 3.1.1]{LR}.

In fact, for any $t, R>0$, $z_i=(x_i, y_i)\in \RR^{n}\times \RR^m$ with $|z_i|\leq R$, $i=1,2$,
\begin{eqnarray*}
&&2\langle \tilde{b}^{\ep}(t,z_1)-\tilde{b}^{\ep}(t,z_2), z_1-z_2\rangle+\|\tilde{\sigma}^{\ep}(t,z_1)-\tilde{\sigma}^{\ep}(t,z_2)\|^2\\
\leq\!\!\!\!\!\!\!\!&&2\langle b(t,x_1, y_1)-b(t,x_2, y_2), x_1-x_2\rangle+\|\sigma(t,x_1)-\sigma(t,x_2)\|^2\\
&&+\frac{2}{\ep}\langle f(t,x_1, y_1)-f(t,x_2, y_2), y_1-y_2\rangle+\frac{1}{\ep}\|g(t,x_1, y_1)-g(t,x_2, y_2)\|^2\\
\leq\!\!\!\!\!\!\!\!&&2| b(t, x_1, y_1)-b(t, x_2, y_1)|\cdot|x_1-x_2|+\|\sigma(t, x_1)-\sigma(t, x_2)\|^2\\
&&+2|b(t, x_2, y_1)-b(t, x_2, y_2)|| x_1-x_2|\\
&&+\frac{2}{\ep}\langle f(t,x_1, y_1)-f(t,x_1, y_2), y_1-y_2\rangle+\frac{1}{\ep}\|g(t,x_1, y_1)-g(t,x_1, y_2)\|^2\\
&&+\frac{1}{\ep}\|g(t,x_1, y_2)-g(t,x_2, y_2)\|^2+\frac{2}{\ep}\|g(t,x_1, y_1)-g(t,x_1, y_2)\|\|g(t,x_1, y_2)-g(t,x_2, y_2)\|\\
&&+\frac{2}{\ep}|f(t,x_1, y_2)-f(t,x_2, y_2)| |y_1-y_2|\\
\leq\!\!\!\!\!\!\!\!&& K_t(R)(1+R^{\theta_1})|x_1-x_2|^2+2(2R^{\theta_2}+K_t(1)+R^{\theta_3})|y_1-y_2|\cdot|x_1-x_2|\\
&&+\frac{C_t}{\ep}(1+2R^{\alpha_1}+R^{\alpha_2})|x_1-x_2||y_1-y_2|+\frac{C_t}{\ep}|z_1-z_2|^2\\
\leq\!\!\!\!\!\!\!\!&& C_R\left[K_t(R)+K_t(1)+\frac{C_t}{\ep}\right]|z_1-z_2|^2.
\end{eqnarray*}
Furthermore, let $\ep_0=\frac{\lambda_2}{\lambda_1}$ if $\lambda_1>0$, and $\ep_0=1$ otherwise. Then for any $\ep\in(0,\ep_0)$
\begin{eqnarray*}
&&2\langle \tilde{b}^{\ep}(t, z_1), z_1\rangle+\|\tilde{\sigma}^{\ep}(t, z_1)\|^2\\
\leq\!\!\!\!\!\!\!\!&&2\langle b(t, x_1, y_1), x_1\rangle+\|\sigma(t, x_1)\|^2+\frac{2}{\ep}\langle f(t,x_1, y_1), y_1\rangle+\frac{1}{\ep}\|g(t,x_1, y_1)\|^2\\
\leq\!\!\!\!\!\!\!\!&&K_t(1)(1+|x_1|^2)+\lambda_1|y_1|^{\theta_4}+K_t(1)+C|x_1|^2-\frac{\lambda_2|y_1|^{\theta_4}}{\ep}+\frac{C_t}{\ep}(1+|x_1|^{\frac{4}{\theta_4}})\\
\leq\!\!\!\!\!\!\!\!&& C\left[2K_t(1)+\frac{C_t}{\ep}\right](1+|z_1|^2).
\end{eqnarray*}
Let
$$K^{\ep}_t(R):=C_R\left[K_t(R)+K_t(1)+\frac{C_t}{\ep}\right].$$
Then by the definition of $K_t(R)$, it is easy to see that $K^{\ep}_t(R)$ is an $\RR_{+}$-valued adapted process and  for all $R, T$, $\ep\in(0,\ep_0)$,
$$
\int^T_0 K^{\ep}_t(R)dt<\infty.
$$
Hence by \cite[Theorem 3.1.1]{LR}, there exists a unique solution $\{(X^{\ep}_t,Y^{\ep}_t), t\geq 0\}$ to system (\ref{Equation}).
The proof is complete.
\end{proof}

\subsection{Proof of Lemma \ref{PMA}}

\begin{proof}
It is sufficient to check that the coefficients of Eq. (\ref{3.1}) satisfy the following conditions:

For any $t\geq0, x_1,x_2\in\RR^n, R>0$ with $|x_i|\leq R$,
\begin{eqnarray}
2|\bar{b}(t,x_1)-\bar{b}(t,x_2)||x_1-x_2|+\|\sigma(t,x_1)-\sigma(t,x_2)\|^2\leq \bar K_t(R)|x_1-x_2|^2\label{4.15}
\end{eqnarray}
and
\begin{eqnarray}
2\langle x_1, \bar{b}(t,x_1)\rangle+\|\sigma(t,x_1)\|^2\leq \bar K_t(1)(1+|x_1|^{2}),\label{4.16}
\end{eqnarray}
where $\bar K_t(R)$ is an $\RR_{+}$-valued adapted process and for all $R, T$>0,
$$
\int^T_0 \bar K_t(R)dt<\infty.
$$
Then Eq.(\ref{3.1}) has a unique solution and (\ref{3.9}) can be easily obtained by following the same arguments as in Lemma \ref{PMY}(i).

In fact, by Proposition \ref{Rem 4.1} and Lemma \ref{L3.7} we have
\begin{eqnarray*}
&&2|\bar{b}(t,x_1)-\bar{b}(t,x_2)||x_1-x_2|+\|\sigma(t,x_1)-\sigma(t,x_2)\|^2\\
\leq\!\!\!\!\!\!\!\!&&2\left|\int_{\RR^m} b(t,x_1,z)\mu^{t,x_1}(dz)-\int_{\RR^m} b(t,x_2,z)\mu^{t,x_2}(dz)\right||x_1-x_2|+\|\sigma(t,x_1)-\sigma(t,x_2)\|^2\\
\leq\!\!\!\!\!\!\!\!&&2\left[\left|\int_{\RR^m} b(t,x_1,z)\mu^{t,x_1}(dz)-\tilde\EE b(t,x_1, Y^{x_1,0}_s)\right|+\left|\tilde \EE b(t,x_2, Y^{x_2,0}_s)-\int_{\RR^m} b(t,x_2,z)\mu^{t,x_2}(dz)\right|\right]|x_1-x_2|\\
&&+2\tilde \EE \left|b(t,x_1, Y^{t,x_1,0}_s)-b(t,x_2, Y^{t,x_1,0}_s)\right||x_1-x_2|+\|\sigma(t,x_1)-\sigma(t,x_2)\|^2\\
&&+2\tilde\EE\left|b(t,x_2, Y^{t,x_1,0}_s)-b(t,x_2, Y^{t,x_2,0}_s)\right||x_1-x_2|\\
\leq\!\!\!\!\!\!\!\!&&C_t e^{-\frac{\beta s}{2}}\left[(K_t(1))^2+1+|x_1|^{(\frac{2\theta_2+2}{\theta_4})\vee(\frac{\theta_3\theta_4+2}{\theta_4})}+|x_2|^{(\frac{2\theta_2+2}{\theta_4})\vee(\frac{\theta_3\theta_4+2}{\theta_4})}\right]\\
&&+C|x_1-x_2|^2 K_t(R)\tilde \EE(1+|Y^{t,x_1,0}_s|^{\theta_1})\\
&&+|x_1-x_2|\tilde \EE\left[(|Y^{t,x_1,0}_s|^{\theta_2}+|Y^{t,x_2,0}_s|^{\theta_2}+|x_2|^{\theta_3}+K_t(1))|Y^{t,x_1,0}_s-Y^{t,x_2,0}_s|\right]\\
\leq\!\!\!\!\!\!\!\!&&C_t e^{-\frac{\beta s}{2}}\left[(K_t(1))^2+1+|x_1|^{(\frac{2\theta_2+2}{\theta_4})\vee(\frac{\theta_3\theta_4+2}{\theta_4})}+|x_2|^{(\frac{2\theta_2+2}{\theta_4})\vee(\frac{\theta_3\theta_4+2}{\theta_4})}\right]\\
&&+C_{t,R} |x_1-x_2|^2 K_t(R)+C_{t,R} (1+K_t(1))|x_1-x_2|^2.
\end{eqnarray*}
Then letting $s\rightarrow \infty$, we obtain
\begin{eqnarray*}
2|\bar{b}(t,x_1)-\bar{b}(t,x_2)||x_1-x_2|+\|\sigma(t,x_1)-\sigma(t,x_2)\|^2\leq  C_{t,R}\left[K_t(R)+K_t(1)+1\right]|x_1-x_2|^2.
\end{eqnarray*}
Moreover,  by \eref{F3.5} we have
\begin{eqnarray*}
&&2\langle\bar{b}(t,x_1), x_1\rangle+\|\sigma(t, x_1)\|^2\\
=\!\!\!\!\!\!\!\!&& \int_{\RR^m} \left[\left\langle 2b(t,x_1,z), x_1\right\rangle+\|\sigma(t, x_1)\|^2\right]\mu^{t,x_1}(dz)\nonumber\\
=\!\!\!\!\!\!\!\!&&\int_{\RR^m} K_t(1)\left(1+|x_1|^2+\lambda_1|z|^{\theta_4}\right)\mu^{t,x_1}(dz)\\
\leq\!\!\!\!\!\!\!\!&&C_t K_t(1)(1+|x_1|^2).
\end{eqnarray*}
Then \eref{4.15} and \eref{4.16} hold by taking
\begin{eqnarray*}
\bar K_t(R):=C_{t,R}\left[K_t(R)+K_t(1)+1\right].
\end{eqnarray*}
By the definition of $K_t(R)$, it is easy to see that $\bar K_t(R)$ is an $\RR_{+}$-valued adapted process and for all $R, T$>0,
$$\int^T_0 \bar K_t(R)dt<\infty. $$
Hence by \cite[Theorem 3.1]{LR}, there exists a unique solution $\{\bar{X}_t, t\geq 0\}$ to Eq. (\ref{3.1}). The estimate \eref{3.9} can be proved by the same arguments as in Lemma \ref{PMY}. Therefore, the proof is complete.
\end{proof}

\vskip 0.2cm

\noindent\textbf{Acknowledgement}
This work was supported in part by NSFC (No. 11571147, 11601196, 11771187, 11822106, 11831014), the NSF of Jiangsu Province
(No. BK20160004),  NSF of the Higher Education Institutions of Jiangsu Province (No. 16KJB110006),  the QingLan Project and the PAPD Project of Jiangsu Higher Education Institutions. Financial support of the DFG through CRC 1283 is also gratefully acknowledged.

\end{document}